\documentclass[12pt,reqno]{amsart}
\usepackage{amsfonts,amstext}
\usepackage{amsmath,amstext,amssymb,amscd,mathrsfs,amsthm}
\usepackage{enumerate}
\usepackage[dvipsnames,usenames]{xcolor}
\newtheorem{thm}{Theorem}[section]
\newtheorem{lem}[thm]{Lemma}
\newtheorem{cor}[thm]{Corollary}
\newtheorem{prop}[thm]{Proposition}

\theoremstyle{definition}
\newtheorem{defn}[thm]{Definition}

\theoremstyle{remark}
\newtheorem{rmk}[thm]{Remark}
\numberwithin{equation}{section}

\usepackage[allbordercolors={1 1 1}]{hyperref}
\newcommand{\rev}{}%
\usepackage[disable]{todonotes}

\def\grad{\nabla}
\newcommand{\W}[1][1,p]{W^{#1}(\Omega)}
\renewcommand{\L}[1][p]{L^{#1}(\Omega)}
\renewcommand{\H}[1][1]{W^{#1,2}(\Omega)}
\newcommand{\Lb}[1][2]{L^{#1}(\Gamma)}
\newcommand{\into}{\hookrightarrow}
\newcommand{\tinto}{\mathrel{\overset{\gamma}{\to}}}
\newcommand{\R}{\mathbb{R}}
\newcommand{\E}{\mathscr{E}}
\let\div\undefined
\DeclareMathOperator{\div}{div}
\DeclareMathOperator{\spn}{span}
\newcommand{\ptag}[1]{ \tag{\ref{#1}$'$} } 
\DeclareMathOperator*{\esssup}{ess\,sup}

\def\G{\Gamma}
\def\O{\Omega}

\def\grad{\nabla}

\newcommand{\reals}{\mathbb{R}}
\newcommand{\naturals}{\mathbb{N}}

\newcommand{\plap}{\Delta_p}

\author[N. J. Kass]{Nicholas J. Kass}
\address{Department of Mathematics, University of Nebraska--Lincoln, Lincoln, NE  68588-0130, USA} \email{nkass@huskers.unl.edu}

\author[M. A. Rammaha]{Mohammad A. Rammaha}
\address{Department of Mathematics, University of Nebraska--Lincoln, Lincoln, NE  68588-0130, USA} \email{mrammaha1@unl.edu}

\thanks{This research was partially supported by NSF grant DMS-1211232.}

\title[Wave equations of the $p$-Laplacian type]{Local and global existence of solutions to a strongly damped wave equation of the $p$-Laplacian type}
\date{November 29, 2017}

\begin{document}

\begin{abstract} This article focuses  on  a quasilinear wave equation of $p$-Laplacian type:
	\[
	u_{tt} - \Delta_p u -\Delta u_t = 0 
	\]
	in a bounded domain $\O \subset \reals^3$ with a sufficiently smooth boundary $\G=\partial \O$ subject to a generalized Robin boundary
	condition featuring boundary damping and a nonlinear source term. The operator $\plap$, $2<p<3$, denotes the classical $p$-Laplacian. The nonlinear boundary term $f(u)$  is a source feedback that  is allowed to  have a \emph{supercritical} exponent, in the sense that the associated Nemytskii operator is not locally Lipschitz from $\W$ into $L^2(\G)$.  
Under suitable assumptions on the parameters we provide a rigorous  proof of existence of a local weak solution which can be extended globally in time provided the source term satisfies an appropriate growth condition. 
\end{abstract}

\maketitle

\section{Introduction}\label{S1}

\subsection{The model} 
This paper is concerned with the  existence of local and global solutions to the quasilinear initial-boundary value problem:
\begin{align}
\label{wave}
\begin{cases}
u_{tt}-\Delta_p u -\Delta u_t = 0 &\text{ in } \Omega \times (0,T),\\[.1in]
(u(0),u_t(0))=(u_0,u_1),\\[.1in]
|\grad u|^{p-2}\partial_\nu u + |u|^{p-2}u + \partial_\nu u_t + u_t = f(u)&\text{ on }\Gamma \times(0,T),
\end{cases}
\end{align}
for given data $(u_0,u_1)\in\W \times \L[2]$ and $2<p<3$. Here, $\Delta_p$ is the $p$-Laplacian given by: 
\begin{align*}
\Delta_p u=\div(|\grad u|^{p-2}\grad u),
\end{align*}
 $\Omega$ is a bounded open domain in $\R^3$ with boundary $\G$ of class $C^2$, and $\partial_\nu$ denotes the outward normal derivative on $\G$.
Additionally, we assume that the boundary source feedback term $f\in C^1(\R)$ is an $\R$-valued function such that  
\begin{align}\label{ass:f}
|f'(u)|\leq C(|u|^{r-1}+1)\text{ where }1\leq r < \frac{4p}{3(3-p)}.
\end{align}
This restriction on the exponent $r$ is inherited from the problem itself and the  Sobolev embedding and trace theorems. In this paper, we focus on the physically relevant case of dimension three, but analogous results in other space  dimensions are possible with corresponding changes in the parameters of the relevant spaces and various Sobolev embeddings.  

\subsection{Literature overview and new contributions}
Strongly damped wave equations of the form
\[
u_{tt} -\Delta u - \Delta u_t = f
\]
 have been given great attention in the literature and are broadly applicable to physical models of damped vibrations.  The damping term $\Delta u_t$ in \eqref{wave} is frequently referred to as Voigt damping in the literature which emphasizes its role in describing so-called Kelvin-Voigt materials exhibiting both elastic and viscous properties.

Important background in damped wave equations is provided by Webb in \cite{MR586981} as well as, for instance, \cite{MR1112054} where $-\Delta$ is replaced by a positive, linear operator. Similarly, \cite{MR797319} contains a more general formulation of a damped wave equation in the form $u_{tt} + Au + Bu_t = f$ with possibly dissimilar operators $A,B$ providing propagation and damping, respectively. 

A sampling of more closely related works which contain a nonlinearity such as the $p$-Laplacian in \eqref{wave} are as follows:
\begin{itemize}
	\item In \cite{MR1634008}, Chen, Guo, and Wang studied equations of the form 
	\begin{align*}
	u_{tt}-\sigma(u_x)_x +f(u)=g(x)
	\end{align*}
	where $\sigma:\R\to\R$ is smooth with $\sigma(0)=0$ and $\sigma'\geq r_0>0$.
	
	\item Biazutti's work in \cite{Bia:95:NA} subsumes the Cauchy problem for the equation 
	\begin{align*}
	 u_{tt}-\Delta_p u -\Delta u_t = f,
	\end{align*}
	while the corresponding problem with zero Dirichl\'{e}t boundary condition and an interior source of supercritical order is treated by Pei et al. in \cite{PRT-p-Laplacain}.
	
	\item Kalantarov et al. in \cite{Bia:95:NA} provide a method in which to establish existence of solutions to the wave equation with structural damping  
	\begin{align*}
	u_{tt} - \Delta u + (\Delta)^\alpha u_t +f(u) = g
	\end{align*} 
	for $\alpha\in (1/2,1]$ and zero Dirichl\'{e}t boundary where $f$ satisfies
	\begin{align*}
	-C+a|s|^q \leq f'(s) \leq C(1+|s|^1)
	\end{align*}
	on $\R$ for positive constants $C,a,q$.
\end{itemize}

The boundary condition imposed in \eqref{wave} is a generalized Robin condition which is not widely  represented in the literature.  A closely related problem is studied  by Vitillaro \cite{V3,V2}:
\begin{align*}
\begin{cases}
u_{tt}-\Delta u=0&\text{ in }\Omega\times (0,T),\\
\partial_\nu u + |u_t|^{m-2}u_t = |u|^{p-2}u&\text{ on }\Gamma\times (0,T).
\end{cases}
\end{align*}

Existence results for problems such as \eqref{wave} have a non-trivial history and are not, in general, amenable to semigroup methods. In this paper, we provide a careful application of the Galerkin method similar to \cite{Bia:95:NA, MR0199519, RW, PRT-p-Laplacain}, amongst others.

In this manuscript, several technical challenges are present, chiefly involving the identification of the limiting value of $\Delta_p u_N$ with the value of $\Delta_p u$ which we carefully accomplish  through the use of monotone operator theory. Our detailed approach also highlights the crucial difficulty that would arise if the Kelvin-Voigt damping were replaced with an $m$-Laplacian term $\Delta_m u_t$, $m>2$. In that case, the simultaneous identification of two weak limits, one for the $p$-Laplacian of $u$ and the other for the $m$-Laplacian of $u_t$ (even if $m=p$) cannot be carried out by the same approach.  It had been assumed in some previous works that the Galerkin approach might trivially extend to the $m$-$p$ model, for instance in \cite{BM2} which attempts to rely on \cite{Bia:95:NA}  and \cite{nak-nan:75} that deal with a \emph{single} $p$-Laplace operator in the equation.
 That is not the case, however, and rigorous analysis of well-posedness for $p$-Laplacian/$m$-Laplacian (with $m,\, p >2$) second-order equation is presently missing from the literature, remaining a challenging open problem.

The literature is quite rich in results on (monotonic) wave equations and systems of wave equations. We mention here the pioneering paper by Lions and Strauss \cite{LSt}, and the important work  by  Glassey 
\cite{G} and Levine \cite{l1}.   Also,  we would like to mention the seminal paper \cite{GT} by Georgiev and Todorova which ignited  immense interest in wave equations influenced by damping and source terms. Subsequent works can be found in \cite{V1} by Vitillaro, and in \cite{CCL} by Cavalcanti et al. In addition to these references, we would like to mention the recent breakthrough papers \cite{BL3,BL2,BL1}  by Bociu and Lasiecka in which they introduced an elegant strategy that deals with supercritical sources. Subsequently, this strategy has been utilized in many recent papers, we mention here \cite{GRSTT,PRT-p-Laplacain,RTW,RW}.  For systems of wave equations in bounded domains, we refer the reader to the papers \cite{AR2,GR1,GR2,GR}.

\subsection{Notation}
Throughout the paper the following notational conventions for $L^p$ space norms and inner products will be used, respectively:
\begin{align*}
&||u||_s=||u||_{\L[s]}, &&|u|_s=||u||_{\Lb[s]};\\
&(u,v)_\Omega = (u,v)_{\L[2]},&&(u,v)_\Gamma = (u,v)_{\Lb[2]}.
\end{align*}
In the interests of clarity, no notational distinction shall be made between a function $u\in\W$ and its trace, typically denoted $\gamma u$, as an element of an appropriate space of functions on $\Gamma$.  It shall be understood that the expression $f(u)$ involving the source term $f$ is to be interpreted as $f(\gamma u)$. 
\\ As is customary, $C$ shall always denote a positive constant which may change from line to line.  Following from the Poincar\'{e}-Wirtinger type inequality
\[
||u||_p^p \leq C(||\grad u||_p^p +|u|_p^p)\text{ for all }u\in \W
\]
we may choose as a matter of convenience
\[
||u||_{1,p}=\left(||\grad u||_p^p +|u|_p^p\right)^{1/p}
\]
as a norm on $\W$ equivalent to the standard norm.\\
For a Banach space $X$, we denote the duality pairing between the dual space $X'$ and $X$ by $\langle \cdot,\cdot \rangle_{X',X}$. That is, 
\begin{align*}
\langle \psi,x \rangle_{X',X} =\psi(x)\text{ for }x\in X,\, \psi\in X'.
\end{align*}
In particular, the duality pairing between $(\W)'$ and $\W$ shall be denoted $\langle \cdot,\cdot\rangle_p$.

By imposing the Robin-type boundary condition $|\grad u|^{p-2}\partial_\nu u + |u|^{p-2}u=0$ on $\Gamma$ the $p$-Laplacian given at the onset of the paper extends readily to a maximal monotone operator from $\W$ into its dual, $(\W)'$, with action given by: 
\begin{align}\label{plapalce}
\langle -\Delta_p u, \phi\rangle_p = \int_\Omega |\grad u|^{p-2}\grad u\cdot\grad\phi\,dx + \int_\Gamma |u|^{p-2}u\phi\,dS,\quad u,\phi\in\W.
\end{align}
Further, it is convenient to record the bound
\begin{align}\label{plaplace-opnorm}
	||-\Delta_p u||_{(\W)'}\leq 2||u||_{1,p}^{p-1}
\end{align}
on the operator norm of $-\Delta_p u$ which follows easily from H\"older's inequality.  As the Laplacian occurs as a term in equation \eqref{wave}  providing damping, it is efficient to utilize all of the preceding notation formally including the case of $p=2$.  Throughout the paper however, we shall always assume $2<p<3$.   Additionally, the Sobolev embedding (in 3D)
\begin{align*}
\W \into \Lb[\frac{2p}{3-p}]
\end{align*}
as well as the inequalities associated with the trace operator $\gamma$ in the map 
\begin{align*}
\W[1-\epsilon,p]\tinto \Lb[\frac{2p}{3-(1-\epsilon)p}]\into\Lb[4]
\end{align*}
for sufficiently small $\epsilon\geq 0$, will be used frequently. (See, e.g., \cite{ADAMS}).
 As it occurs so frequently we shall pass to subsequences consistently without re-indexing.
\begin{rmk}\label{rmk:fbound}
	As the bound will be used often throughout the paper it is worthy of note that the assumptions on $f$ imply that $|f(u)|\leq C(|u|^r+1), \,\, u\in \reals$.
\end{rmk}

\subsection{Main results and strategies}
A suitable weak formulation of \eqref{wave} is as follows:
\begin{defn}\label{def:weaksln}  A function $u$ is said to be a weak solution of \eqref{wave} on the interval $[0,T]$ provided:
	\begin{enumerate}[(i)]
		\setlength{\itemsep}{5pt}
		\item\label{def-a} $u\in C_w([0,T];\W)$,
		\item\label{def-b} $u_t\in L^2(0,T;\W[1,2])\cap C_w([0,T];\L[2]),$
		\item\label{def-c} $(u(0),u_t(0))=(u_0,u_1)$ in $\W \times \L[2]$,
		\item\label{def-d} and for all $t\in[0,T]$ the function $u$ verifies the identity
			\rev\begin{multline}\label{slnid}
				(u_t(t),\phi(t))_\O - (u_1,\phi(0))_\O - \int_0^t (u_t(\tau),\phi_t(\tau))_\O\,d\tau 
				+ \int_0^t \langle -\Delta_p u(\tau),\phi(\tau)\rangle_p\,d\tau \\
				+ \int_0^t \langle -\Delta_2 u_t(\tau),\phi(\tau)\rangle_2\,d\tau 
				= \int_0^t \int_\G f(u(\tau))\phi(\tau)\,dSd\tau 
			\end{multline}		
	\end{enumerate}
for all test functions $\phi\in C_w([0,T];\W)$ with $\phi_t\in L^2(0,T;\W[1,2])$.		
\end{defn}
\begin{rmk}
	In Definition~\ref{def:weaksln} above, $C_w([0,T]; X)$ denotes the space of weakly continuous (often called scalarly continuous) functions from $[0,T]$ into a Banach space $X$.  That is, for each $u\in C_w([0,T];X)$ and  $f \in X'$ the map $t\mapsto \langle f, u(t) \rangle_{X',X}$ is continuous on $[0,T]$.
\end{rmk}
The principal result is the existence of local solutions of problem \eqref{wave} in the following sense.
\begin{thm}[\bf Local solutions]\label{thm:exist}
	Under the stated assumptions, problem \eqref{wave} possesses a local weak solution, $u$, in the sense of Definition~\ref{def:weaksln} on a non-degenerate interval $[0,T]$ with length dependent only upon the initial data, $(u_0,u_1)$, and the local Lipschitz constant of $f$ as a map from $\W$ into $\Lb[4/3]$. Further,  this solution $u$ satisfies the energy inequality
	 \begin{align}\label{energy-2ndid}
		\E(t) + \int_0^t ||u'(\tau)||_{1,2}^2\,d\tau \leq\E(0) + \int_0^t \int_\Gamma f(u(\tau))u'(\tau)\,dSd\tau
		\end{align}
		where $\E(t)=\frac{1}{2}||u'(t)||_2^2 + \frac{1}{p}||u(t)||_{1,p}^p$.
	 Equivalently, \eqref{energy-2ndid} can also be written as
	 \begin{align}\label{energy-1stid}
			E(t) + \int_0^t ||u'(\tau)||_{1,2}^2\,d\tau \leq E(0)
		\end{align}
		with $E(t)=\frac{1}{2}||u'(t)||_2^2 + \frac{1}{p}||u(t)||_{1,p}^p -\int_{\Gamma} F(u(t))\,dS$
	by taking  $F$ as the primitive of $f$, i.e.  $F(u)=\int_0^u f(s)\,ds$.	
	 \end{thm}
\begin{rmk}
Note that no claims of uniqueness are made here.
\end{rmk}

Our next Theorem states that the weak solution described by Theorem \ref{thm:exist} can be extended globally in time provided the source exponent is at most $p/2$.

\begin{thm} [\bf Global solutions]\label{t-2}In addition to the assumptions of Theorem \ref{thm:exist}  assume that $r \leq p/2$.
Then, the  weak solution $u$ furnished by Theorem \ref{thm:exist} is a global solution and the existence time $T$ may be taken arbitrarily large.
\end{thm}

The proof of Theorem \ref{thm:exist} is accomplished in several steps similar to \cite{GR, PRT-p-Laplacain,RW}.   In Section~\ref{S2} we construct a solution satisfying Theorem~\ref{thm:exist} using  Galerkin approximations under the added assumption that $f:\W \to \Lb[2]$ is globally Lipschitz continuous.    This permits us to focus on the recovery of the weak limit of the terms due to the $p$-Laplacian in a case where the behavior of $f$ is relatively benign.  

In Section~\ref{S3} we extend the results first to sources $f:\W \to \Lb[2]$ which are locally Lipschitz  by a standard  truncation argument similar to \cite{CEL1, KL} while ensuring that the interval of existence for these solutions depends only upon the local Lipschitz constants of $f$ as a map into $\Lb[4/3]$.  A further truncation argument  permits the full range exponents on $f$ prescribed in \eqref{ass:f} by constructing a sequence of approximated solutions whose truncated sources, $f_n$, have uniformly bounded local Lipschitz constants as maps into $\Lb[4/3]$.

Finally, in Section~\ref{S4} we  provide the proof of Theorem~\ref{t-2}  by obtaining the appropriate bounds on $u$ and $u_t$ and appealing to a standard continuation procedure.

\subsection{Preliminaries}

\sloppy While the inequalities associated with the  trace mapping $\W \tinto \Lb[2]$ ensures that $f:\W \to \Lb[2]$ is locally Lipschitz for $1\leq r< p/(3-p)$, a stronger result provided by the following lemma will be used frequently. Its chief utility is in conjunction with the compactness results from Section~\ref{S2}.

\begin{lem}\label{lem:f-lipshitz}
	Under the assumptions given in \eqref{ass:f}, the function $f:\W[1-\epsilon,p]\to \Lb[4/3]$ is locally Lipschitz continuous for sufficiently small $\epsilon\geq 0$.
\end{lem}
\begin{proof}
	Selecting $R>0$ it is enough to find a constant $C_R$ so that 
	\begin{align*}
	|f(u)-f(v)|_{4/3}^{4/3} \leq C_R||u-v||_{1-\epsilon,p}^{4/3}
	\end{align*}
	for all $u,v\in\W[1-\epsilon,p]$ with $||u||_{1-\epsilon,p},||v||_{1-\epsilon,p}\leq R$.
	By the mean value theorem applied to $f$ we may produce the bound
	\begin{align*}
	|f(u)-f(v)| &\leq |f'(\xi_{u,v})||u-v| \\
	&\leq C(|u|+|v|+1)^{r-1}|u-v|.
	\end{align*}
	Select $\alpha=\frac{3}{4}\frac{2p}{3-(1-\epsilon)p}$ with an eye towards the  trace mapping   $\W[1-\epsilon,p]\tinto \Lb[4\alpha/3]$. Since $\alpha$ increases to $\frac{3p}{2(3-p)}$ as $\epsilon$ decreases towards zero we may further select $\epsilon$ sufficiently small so that $r<\alpha$ following from \eqref{ass:f}.  This choice of $\epsilon$ implies that $\frac{r-1}{\alpha-1}<1$ and thus $\frac{4\alpha(r-1)}{3(\alpha -1)}<\frac{4}{3}\alpha$, whereupon utilizing H\"{o}lder's inequality we obtain, as desired,    
	\begin{align*}
	|f(u)-f(v)|_{4/3}^{4/3} &\leq C|u+v+1|_{4\alpha(r-1)/3(\alpha-1)}^{(4r-4)/3}|u-v|_{4\alpha /3}^{4/3}\\
	&\leq C|u+v+1|_{4\alpha /3}^{(4r-4)/3}|u-v|_{4\alpha /3}^{4/3}\\
	&\leq C(||u||_{1-\epsilon,p}^{(4r-4)/3} + ||v||_{1-\epsilon,p}^{(4r-4)/3} + 1)||u-v||_{1-\epsilon,p}^{4/3} \\
	&\leq C_R ||u-v||_{1-\epsilon,p}^{4/3}
	\end{align*}
	with $C_R=C(2R^{(4r-4)/3}+1)$.
\end{proof}

\section{Solutions for globally Lipschitz sources}\label{S2}
Our strategy is to employ a suitable Galerkin approximation to show the local existence of weak solutions of \eqref{wave} satisfying the conditions of Theorem~\ref{thm:exist} in the case where $f:\W \to \Lb[2]$ is globally Lipschitz.

\subsection{Approximate solutions}\label{S2.1}
To start, let $\{w_j\}_1^\infty$ be a basis for $\W$ suitably normalized so as to form an orthonormal basis for $\L[2]$. (Such a basis may be constructed from, for instance, eigenfunctions of the Laplacian $\mathscr{L}=-\Delta$ viewed as an unbounded, positive, self-adjoint operator with domain $\mathscr{D}(\mathscr{L})=\{w\in H^2(\Omega):\partial_\nu w+w=0\text{ on }\Gamma\}$, whose inverse is a compact operator.)  We now seek to construct a sequence of approximate solutions in the form 
\begin{align}\label{approxform}
u_N(x,t)=\sum_{j=1}^N u_{N,j}(t)w_j(x) 
\end{align} 
on each of the finite dimensional subspaces $V_N=\spn\{w_1,\ldots,w_N\}$ of $\W$.\\
As the collection of functions $\{w_j\}_1^\infty$ forms a basis both for $\W$ and $\L[2]$, we may find corresponding to each $N$ sequences of scalars $\{u^0_{N,j}\}_{N,j=1}^\infty$ and $\{u^1_j=(u_1,w_j)_\Omega\}_1^\infty$ such that 
\begin{subequations}\label{approxIC}\begin{alignat}{5}
&\sum_{j=1}^N u^0_{N,j} w_j&\to u_0&\text{ strongly in }&&\W, \\
\mathscr{P}_Nu_1 =&\sum_{j=1}^N u^1_j w_j&\to u_1&\text{ strongly in }&&\L[2]
\end{alignat}\end{subequations}
for given initial data $(u_0,u_1)\in\W\times\L[2]$, where $\mathscr{P}_N:L^2(\O) \rightarrow V_N$ denotes orthogonal projection onto $V_N$.  Now, for each $N=1,2,\ldots$ the system of ordinary differential equations:
\begin{subequations}
\begin{multline}\label{approx1}
(u_N'',w_j)_\Omega +(|\grad u_N|^{p-2}\grad u_N,\grad w_j)_\Omega + (|u_N|^{p-2}u_N,w_j)_\Gamma \\\rev
 +(\grad u_N',\grad w_j)_\Omega + (u_N',w_j)_\Gamma = (f(u_N),w_j)_\Gamma
\end{multline} 
indexed by $j=1,\ldots,N$ with initial conditions 
\begin{align}\label{approx2}\rev
u_{N,j}(0)=u^0_{N,j},\,\,\,u'_{N,j}(0)=u^1_j
\end{align}
\end{subequations}
is an initial value problem for a second order $N\times N$ system of ordinary differential equations with continuous nonlinearities in the unknown functions $u_{N,j}$ and their time derivatives. Therefore, it follows from the Cauchy-Peano theorem that for every $N\geq 1$, \eqref{approx1}--\eqref{approx2} has a solution $u_{N,j}\in C^2([0,T_N])$, $j=1,\ldots N$, for some $T_N>0$.
\subsection{A priori estimates} We next demonstrate that each of the approximate solutions $u_N$ exists on a non-degenerate interval $[0,T]$ independent of $N$. 
\begin{prop}\label{prop:apriori} Each approximate solution $u_N$ exists on $[0,\infty)$.
Further, for any $T>0$ the sequence of approximate solutions $\{u_N\}_1^\infty$ satisfies 
	\begin{subequations}\begin{align}
	\{u_N\}_1^\infty&\text{ is a bounded sequence in }L^{\infty} (0,T;\W), \label{apriori-a}\\
	\{u_N'\}_1^\infty&\text{ is a bounded sequence in }L^{\infty} (0,T;\L[2]),\label{apriori-b}\\
	\{u_N'\}_1^\infty&\text{ is a bounded sequence in }L^2(0,T;\H[1]),\label{apriori-c}\\
	\{u_N''\}_1^\infty&\text{ is a bounded sequence in }L^2(0,T;(\W)').\label{apriori-d}
	\end{align}\end{subequations}
\end{prop}
\begin{proof}
	Multiplying equation \eqref{approx1} by $u'_{N,j}$ and summing over $j=1,\ldots,N$ we obtain the relation 
\begin{align}\label{apriori-1}
\frac{1}{2}\frac {d}{dt}||u_N'(\tau)||_2^2 + \frac{1}{p}\frac{d}{dt}||u_N(\tau)||_{1,p}^p + ||u_N'(\tau)||_{1,2}^2 = \int_\Gamma f(u_N)u_N'\,dS
\end{align}
for each $\tau\in[0,T_N]$. 

Integrating \eqref{apriori-1} over $\tau\in[0,t]$,  we obtain
\begin{align}\label{apriori-2}
\frac{1}{2}||u_N'(t)||_2^2 &+ \frac{1}{p}||u_N(t)||_{1,p}^p + \int_0^t||u_N'(\tau)||_{1,2}^2\,d\tau \nonumber\\
&=  \int_0^t\int_\Gamma f(u_N)u_N'\,dSd\tau 
+ \frac{1}{2}||u_N'(0)||_2^2 + \frac{1}{p}||u_N(0)||_{1,p}^p 
\ptag{apriori-1}
\end{align}
at each $t\in[0,T_N]$. From this, it is natural to define $\epsilon_N$ by 
\begin{align*}
\epsilon_N(t)=\frac{1}{2}||u_N'(t)||_2^2 + \frac{1}{p}||u_N(t)||_{1,p}^p
\end{align*}
so that for all $t\in[0,T_N]$ we may express \eqref{apriori-2} equivalently as  
\begin{align}\label{apriori-3}
\epsilon_N(t) + \int_0^t ||u_N'(\tau)||_{1,2}^2\,d\tau = \int_0^t\int_\Gamma f(u_N(\tau))u_N'(\tau)\,dSd\tau+\epsilon_N(0).\ptag{apriori-2}
\end{align}

In order to produce a suitable bound on $\epsilon_N$ we first address the term due to the source.  Under the assumption that $f:\W \to \Lb[2]$ is globally Lipschitz we have the bound $|f(u_N)|_2 \leq C(1+||u_N||_{1,p})$ with, say, $C$ greater than $|f(0)|_2$ and the Lipschitz constant of $f$. Utilizing H\"older and Young's inequalities, 
\begin{align}\label{apriori-3.5}
\int_\Gamma f(u_N)u_N'\,dS &\leq C_\epsilon|f(u_N)|_2^2 + \epsilon|u_N'|_2^2\nonumber \\
&\leq C(1+||u_N||_{1,p}^2) + \frac{1}{2}||u_N'||_{1,2}^2
\end{align}
 for an appropriate choice of $\epsilon$. Since $1+||u_N(\tau)||_{1,p}^2\leq 2 +||u_N(\tau)||_{1,p}^p$ we may apply the bound in \eqref{apriori-3.5} to equation \eqref{apriori-3} so that  
\begin{align}\label{apriori-4}
\epsilon_N(t)+\frac{1}{2}\int_0^t||u_N'(\tau)||_{1,2}^2\,d\tau &\leq C\int_0^t \left( 1+ \epsilon_N(\tau) \right)\,d\tau + \epsilon_N(0).
\end{align}
In particular, this means that each $\epsilon_N$ satisfies the integral inequality 
\begin{align*}
\epsilon_N(t)\leq C\int_0^t \left( 1+ \epsilon_N(\tau) \right)\,d\tau + C',
\end{align*} 
with constants $C,\,C'$ depending only upon the Lipschitz constant of $f$, $p$, and the values of $||u_0||_{1,p}$ and $||u_1||_2$.

 Using Gronwall's inequality we may thus for any $T>0$ bound each $\epsilon_N$ on the interval $[0,T]$ upon which each approximate solution $u_N$ exists as a consequence of the Cauchy-Peano theorem. This bound on $\epsilon_N$ also establishes \eqref{apriori-a} and \eqref{apriori-b}, with \eqref{apriori-c} following immediately from \eqref{apriori-4}.  

For the final claim, given any $\phi\in \W[1,p]$ we may select by density a sequence $\{\phi_j\}_1^\infty$ with each $\phi_j$ expressed as 
\begin{align*}
\phi_j = \sum_{i=1}^{j} a_i^jw_i;\quad\{a_i^j\}\subset\R
\end{align*}
so that $\phi_j \to \phi$ strongly in $\W[1,p]$. We may then use \eqref{approx1} along with H\"older's inequality to produce the bound 
\begin{align}\label{apriori-5}
  |\langle u_N'', \phi_j \rangle_p| &= |(u_N'',\phi_j)_\Omega| \notag\\
   &\leq ||\grad u_N||_p^{p-1} ||\grad\phi_j||_p  + |u_N|_p^{p-1}|\phi_j|_p \notag\\ 
  &\qquad + ||\grad u_N'||_2||\grad\phi_j||_2  + |u_N'|_2|\phi_j|_2 + |f(u_N)|_2|\phi_j|_2\notag\\
  &\leq C\left(||u_N||_{1,p}^{p(p-1)} +||u_N'||_{1,2} + |f(u_N)|_2\right)||\phi_j||_{1,p}\notag \\
  & \leq C\left(1+||u_N||_{1,p}^{p(p-1)} +||u_N'||_{1,2} + ||u_N||_{1,p} \right)||\phi_j||_{1,p}
\end{align}
since, as before, $|f(u_N)|_2 \leq C(1+||u_N||_{1,p})$.  By integrating the square of \eqref{apriori-5} the desired conclusion of \eqref{apriori-d} follows immediately from \eqref{apriori-a} and \eqref{apriori-c} which assert that each of $||u_N||_{1,p}$ and $||u_N'||_{1,2}$ is bounded in $L^2(0,T)$.
\end{proof}

An immediate consequence of Proposition~\ref{prop:apriori} along with the Banach-Alaoglu theorem and the standard Aubin-Lions-Simon compactness theorems (e.g., \cite[Thm. II.5.16]{Boyer2013}) is the following:
 \begin{subequations}
\begin{cor}\label{cor:converg} For all sufficiently small $\epsilon>0$ there exists a function $u$ and a subsequence of $\{u_N\}$ (still denoted $\{u_N\}$) such that  
	 \begin{alignat}{2}
	 	u_N&\to u &\text{ weak* in }&L^\infty(0,T;\W), \label{converg:a} \\
		u_N'&\to u' &\text{ weak* in }&L^\infty(0,T;\L[2]),\label{converg:b} \\
		u_N'&\to u' &\text{ weakly in }&L^2(0,T;\H[1]),\label{converg:c} \\
		u_N &\to u &\text{ strongly in }&C([0,T];\W[1-\epsilon,p]),\label{converg:d} \\
		u_N' &\to u' &\text{ strongly in }&L^2(0,T;\W[1-\epsilon,2]),\label{converg:e}\\
		u_N''&\to u'' &\text{ weakly in }&L^2(0,T;(\W)').\label{converg:g}
	\end{alignat}
\end{cor}

 By utilizing a routine density argument we also obtain convergence in the following sense:
\begin{cor}
	On a subsequence, 
	\begin{alignat}{2}
	u_N(t)&\to u(t)&\text{ weakly in }&\W\text{ for a.e. }t\in[0,T].\label{converg:f}
	\end{alignat}
\end{cor}
\begin{proof}
	From \eqref{converg:a} and \eqref{converg:d} we obtain $u_N\to u$ weakly in $L^1(0,T;\W)$ and $u_N\to u$ strongly in $L^1(0,T;\L[2])$, respectively.  Since each of the embeddings in the map $\W\into \L[2]\into (\W)'$ is continuous with dense range and $(\W)'$ is separable given that $\W$ is reflexive and separable, the desired conclusion follows by \cite[Proposition~A.2]{PRT-p-Laplacain}.
\end{proof}
\end{subequations}

\subsection{Passage to the limit}\label{S2-lim} By integrating \eqref{approx1} on $[0,t]$ it is seen that each approximate solution $u_N$ satisfies the identity 
\begin{multline}\label{limit-1}
		\int_0^t \langle u_N''(\tau),w_j\rangle_p \,d\tau  
		+ \int_0^t \langle -\Delta_p u_N(\tau), w_j\rangle_p \,d\tau
		+ \int_0^t \langle -\Delta_2 u_N'(\tau),w_j\rangle_2 \,d\tau \\
		= \int_0^t \int_\G f(u_N(\tau))w_j \,dS\,d\tau
\end{multline}
for $j=1,\ldots,N$.  As a first step in demonstrating that the limit function $u$ indeed verifies the identity \eqref{slnid} we shall first carefully pass to the limit as $N\to\infty$ in \eqref{limit-1}.  This process is routine for most of these terms with the principle difficulty being the recovery of the terms due to the $p$-Laplacian.  Precisely, while it is not difficult to show that $-\Delta_p u_N$ must converge to $\eta$ in some sense it is difficult to demonstrate that, in fact, $\eta=-\Delta_p u$.  We address this issue by appealing to results from monotone operator theory using an argument analogous to \cite{PRT-p-Laplacain} in the course of Proposition~\ref{prop:limlaplace}.

First, however, we record the crucial but easily addressed matters of convergence of the source term, the initial conditions, and the terms due to the damping. The first of these is an immediate consequence of Lemma~\ref{lem:f-lipshitz} and the convergence \eqref{converg:d}.

\begin{prop}\label{prop:limf} With $\{u_N\}$ and $u$ as in Corollary~\ref{cor:converg},
	\begin{align*}
		f(u_N)\to f(u)\text{ strongly in }L^\infty(0,T;\Lb[4/3]).
	\end{align*}
\end{prop}
\begin{proof}
 From Lemma~\ref{lem:f-lipshitz} and \eqref{converg:d} in Corollary~\ref{cor:converg},
 \begin{align*}
 |f(u_N(t))-f(u(t))|_{4/3}\leq C||u_N(t)-u(t)||_{1-\epsilon,p}\to 0;\,t\in [0,T].
 \end{align*}
\end{proof}

As each approximate solution was constructed to satisfy the initial conditions given in \eqref{approx2}, the following which follows immediately from the convergence given in \eqref{approxIC}.

\begin{prop}\label{prop:ICs}
	The approximate solutions $\{u_N\}$ satisfy 
	\begin{align*}
	(u_N(0),u_N'(0)) \to (u_0,u_1)\text{ strongly in }\W\times\L[2].
	\end{align*}
	In particular, the limit function $u$ identified in Corollary~\ref{cor:converg} satisfies \begin{align*}
	(u(0),u_t(0))=(u_0,u_1)\text{ in }\W\times\L[2].
	\end{align*}
\end{prop}

Additionally, the convergence of the terms arising from the damping are fairly easy to justify oweing to the linearity of the Laplacian as recorded in the following proposition.

\begin{prop}\label{prop:dampconverg} With $\{u_N\}$ and $u$ as in Corollary~\ref{cor:converg}, 
	\begin{align*}
	-\Delta_2 u_N' \to -\Delta_2 u'\text{ weakly in }L^2(0,T;(\W[1,2])').	
	\end{align*}
\end{prop}
\begin{proof}For $\phi\in L^2(0,T;\W[1,2])$ we have 
	\begin{align*}
	\int_0^T \langle -\Delta_2 u_N'(\tau),\phi(\tau)\rangle_2\,d\tau = 
	\underbrace{\int_0^T \int_\O \grad u_N'(\tau)\cdot\grad\phi(\tau)\,dxd\tau}_\text{(i)} + 
	\underbrace{\int_0^T \int_\G u_N'(\tau)\phi(\tau)\,dSd\tau}_\text{(ii)}.	
	\end{align*}
	For $\text{(i)}$, we have $\grad u_N' \to \grad u'$ weakly in $L^2(0,T;\L[2])$ from \eqref{converg:c} with $\grad\phi\in L^2(0,T;\L[2])$; and for $\text{(ii)}$, we use the fact that $u_N' \to u'$ strongly in $L^2(0,T;\Lb[2])$ from \eqref{converg:e} along with the continuity of the map $\W[1-\epsilon,2]\tinto \Lb[2]$ for sufficiently small $\epsilon>0$.  Thus,
	\begin{align*}
	\int_0^T \langle -\Delta_2 u_N'(\tau),\phi(\tau)\rangle_2 \,d\tau &\to  
	\int_0^T \int_\O \grad u'(\tau)\cdot\grad\phi(\tau)\,dxd\tau
	+ \int_0^T \int_\G u'(\tau)\phi(\tau)\,dSd\tau\\
	&= \int_0^T \langle -\Delta_2 u'(\tau),\phi(\tau)\rangle_2\,d\tau.
	\end{align*}
\end{proof}

We are now in a position to address the more difficult terms arising from the $p$-Laplacian. 

\begin{prop}\label{prop:limlaplace}
	Up to a subsequence, the sequence of approximate solutions satisfies 
	\begin{align*}
	-\Delta_pu_N \to -\Delta_p u\text{ weak* in }L^\infty(0,T;(\W)'). 
	\end{align*}
\end{prop}
\begin{proof}
	Throughout, set $X=L^p(0,T;\W)$. 
 By utilizing the operator norm bound on $-\Delta_p$ from \eqref{plaplace-opnorm} we see that 
	 	\begin{align*}
	 		||-\Delta_p u_N||_{L^\infty(0,T;(\W)')} &= \esssup_{\tau\in[0,T]}||-\Delta_p u_N(\tau)||_{(\W)'}\\
	 		&\leq 2\esssup_{\tau\in[0,T]}||u_N(\tau)||_{1,p}^{p-1}<\infty 
	 	\end{align*}
	 	by virtue of \eqref{apriori-a}, whereupon the sequence $\{-\Delta_p u_N\}$ is seen to be bounded in $L^\infty(0,T;(\W)')$.  As such there exists some $\eta \in L^\infty(0,T;(\W)')$ and a subsequence of $\{u_N\}$ so that  
	 	\begin{align}\label{limlaplace-1}
	 		-\Delta_p u_N\to \eta\text{ weak* in }L^\infty(0,T;(\W)').  
	 	\end{align}
	 	By viewing $L^\infty(0,T;(\W)')$ as a subspace of $X'$ the desired conclusion follows immediately by demonstrating that $-\Delta_p u_N \to -\Delta_p u$ weakly in $X'$.
	 	We first note that the $p$-Laplacian extends to a maximal monotone operator on $X$ with action 
	 	\begin{align*}
	 		\langle -\Delta_p u,\phi\rangle_{X',X} = \int_0^T \langle -\Delta_p u(\tau),\phi(\tau)\rangle_p \,d\tau;\qquad u,\phi\in X
	 	\end{align*}  
	 	in accordance with Lemma~\ref{lem:plaplace-mmono}.  Using a standard result from monotone operator theory (see \cite{Barbu2010}, for instance) we may conclude that $\eta = -\Delta_p u$ in $X'$  provided 
	\begin{align*}
	\limsup_{N\to\infty} \langle -\Delta_p u_N -\eta, u_N-u\rangle_{X',X}\leq 0.
	\end{align*} 
	Since from Corollary~\ref{cor:converg} we have $\langle \eta, u_N-u\rangle_{X',X}\to 0$ whereas $L^\infty(0,T;\W)\subset X$ and $\langle -\Delta_p u_N , u\rangle_{X',X} \to \langle \eta, u\rangle_{X',X}$ from \eqref{limlaplace-1} it is in turn enough to demonstrate that 
	\begin{align}\label{limlaplace-wts}\rev
	\limsup_{N\to\infty}\langle -\Delta_p u_N,u_N\rangle_{X',X} \leq \langle  \eta, u\rangle_{X',X}.
	\end{align}
	
	Multiplying equation \eqref{approx1} by $u_{N,j}$ and summing over $j=1,\ldots,N$ we obtain the relation 
	\begin{align}\label{limlaplace-s1}
	(u_N'',u_N)_\Omega + \langle -\Delta_p u_N,u_N\rangle_p + (\grad u_N',\grad u_N)_\Omega + (u_N',u_N)_\Gamma = (f(u_N),u_N)_\Gamma. 
	\end{align}
	Rearranging \eqref{limlaplace-s1} and integrating over $[0,t]$ we thus obtain 
	\begin{multline}\label{limlaplace-s1p}
	\int_0^t \langle -\Delta_p u_N,u_N\rangle_p \,d\tau = -\int_0^t (u_N'',u_N)_\Omega\,d\tau -\int_0^t (\grad u_N',\grad u_N)_\Omega\,d\tau \\
	  -\int_0^t (u_N',u_N)_\Gamma \,d\tau +\int_0^t \int_\G f(u_N)u_N \,dS \,d\tau. \ptag{limlaplace-s1}
	\end{multline}
	Thus, upon integrating by parts we may write  
	\begin{multline}\label{limlaplace-s1pp}
	\int_0^t \langle -\Delta_p u_N,u_N\rangle_p \,d\tau = 
		\underbrace{(u_N'(0),u_N(0))_\Omega
		 - (u_N'(t),u_N(t))_\Omega}_\text{(i)} \\ 
		  + \underbrace{\int_0^t ||u_N'(\tau)||_2^2\,d\tau}_\text{(ii)} 
		  - \underbrace{\frac{1}{2}||\grad u_N(t)||_2^2 + \frac{1}{2}||\grad u_N(0)||_2^2}_\text{(iii)}\\
		  -\underbrace{\int_0^t (u_N',u_N)_\Gamma \,d\tau}_\text{(iv)}
		   +\underbrace{\int_0^t \int_\G f(u_N)u_N \,dS \,d\tau}_\text{(v)}. \ptag{limlaplace-s1p}
	\end{multline}
	The convergence of these terms warrants special attention:
	\begin{enumerate}[(i)]
				\setlength{\itemsep}{5pt}
	
	\item \sloppy From Proposition~\ref{prop:ICs},  $(u_N'(0),u_N(0))_\Omega\to(u_1,u_0)_\Omega$.
	Using \eqref{converg:d} in Corollary~\ref{cor:converg} we obtain  $||u_N-u||_2\to 0$ in $L^2(0,T)$, and hence on a subsequence $u_N(t)\to u(t)$ strongly in $\L[2]$, a.e. $t\in[0,T]$. Similarly, from \eqref{converg:e} we find $u_N'(t)\to u'(t)$ strongly in $\L[2]$ for a.e. $t\in[0,T]$ on a subsequence.  Thus, $(u_N'(t),u_N(t))_\Omega\to (u'(t),u(t))_\Omega$ for a.e. $t\in[0,T]$ on a common subsequence of $\{u_N\}$ and $\{u_N'\}$.

	\item Since $u_N'\to u'$ strongly in $L^2(0,T;\L[2])$ from \eqref{converg:e} in Corollary~\ref{cor:converg},  then it follows that
	\begin{align*}
	\int_0^t ||u_N'(\tau)||_2^2\,d\tau \to \int_0^t ||u'(\tau)||_2^2\,d\tau.
	\end{align*}

	\item    From \eqref{converg:f} we have $\grad u_N(t)\to \grad u(t)$ weakly in $\L[2]$ for a.e. $t\in[0,T]$, whereby 
	using the weak lower semicontinuity of norms we obtain \begin{align*}
		\limsup_{N\to\infty}-\frac{1}{2}||\grad u_N(t)||_2^2 
		= -\frac{1}{2}\liminf_{N\to\infty}||\grad u_N(t)||_2^2 
		\leq -\frac{1}{2}||\grad u(t)||_2^2;\quad \text{a.e. }[0,T].
		\end{align*}
	At $t=0$ we find $\grad u_N(0) \to \grad u(0)$ strongly in $(L^p(\Omega))^3$ from Proposition~\ref{prop:ICs}, and thus $||\grad u_N(0)||_2^2 \to ||\grad u(0)||_2^2$.

\item From \eqref{converg:d} in Corollary~\ref{cor:converg} and the embedding $\W[1-\epsilon,p] \into \Lb[2]$ we have up to a subsequence $u_N(t)\to u(t)$ strongly in $C([0,T];\Lb[2])$, and similarly from \eqref{converg:c} and the embedding $\W[1,2]\into \Lb[2]$ we obtain $u_N'\to u'$ weakly in $L^2(0,T;\Lb[2])$. Hence, $u_Nu_N'\to uu'$ weakly in $L^1(\Gamma \times (0,T))$. By using  the indicator function $\chi_{[0,t]}$ one has 
	\begin{align*}
		\int_0^t (u_N',u_N)_\Gamma\,d\tau \to \int_0^t (u',u)_\Gamma\,d\tau.
	\end{align*}

	\item Since $f(u_N)\to f(u)$ strongly in $L^\infty(0,T;\Lb[4/3])$ by Proposition~\ref{prop:limf} and $u_N\to u$ strongly in $C([0,T];\Lb[4])$ from Corollary~\ref{cor:converg} and the embedding $\W[1-\epsilon,p]\to \Lb[4]$,  therefore for any $t\in [0,T]$ we have 
	$$\int_0^t \int_\G f(u_N)u_N \,dS \,d\tau \to  \int_0^t \int_\G f(u)u \,dS \,d\tau. $$
		\end{enumerate}
We may thus take the limit  superior as $N\to\infty$ in \eqref{limlaplace-s1pp} to obtain 
 \begin{multline}\label{limlaplace-s2}
\limsup_{N\to\infty} \int_0^t  \langle -\Delta_p u_N,u_N\rangle_p\,d\tau 
\leq (u'(0),u(0))_\Omega - (u'(t),u(t))_\Omega  \\ 
+ \int_0^t ||u'(\tau)||_2^2\,d\tau  
-\frac{1}{2}||\grad u(t)||_2^2 + \frac{1}{2}||\grad u(0)||_2^2\\  
-\int_0^t (u',u)_\Gamma \,d\tau + \int_0^t \int_\G f(u)u \,dS \,d\tau \,\,\, \text{ a.e. }t\in[0,T].
\end{multline}
In order to express the right hand side of \eqref{limlaplace-s2} in terms of $\eta$ we utilize the separable nature of the approximate solutions to effect a limit of \eqref{limlaplace-s1pp} through a different means. Towards these ends, multiplying \eqref{approx1} by any $\phi\in C^1([0,T])$ and integrating on $[0,t]$ yields 
\begin{multline}\label{limlaplace-s3}
	\int_0^t \langle -\Delta_p u_N,\phi w_j\rangle_p \,d\tau = 
	(u_N'(0),\phi(0)w_j)_\Omega - (u_N'(t),\phi(t)w_j)_\Omega \\ 
	+ \int_0^t (u_N',\phi'w_j)_\Omega\,d\tau 
	-\int_0^t (\grad u_N',\phi\grad w_j)_\Omega\,d\tau \\ 
	-\int_0^t (u_N',\phi w_j)_\Gamma \,d\tau +\int_0^t \int_\G f(u_N)\phi w_j\, dS \,d\tau.
\end{multline}
From the convergence given in Corollary~\ref{cor:converg} and using arguments analogous to those used in obtaining \eqref{limlaplace-s2} we obtain by taking the limit in \eqref{limlaplace-s3} as $N\to\infty$ that
\begin{multline}\label{limlaplace-s4}
\int_0^t \langle \eta,\phi w_j\rangle_p \,d\tau = 	(u'(0),\phi(0)w_j)_\Omega - (u'(t),\phi(t)w_j)_\Omega \\ 
+ \int_0^t (u',\phi'w_j)_\Omega\,d\tau 
-\int_0^t (\grad u',\phi\grad w_j)_\Omega\,d\tau \\ 
-\int_0^t (u',\phi w_j)_\Gamma \,d\tau +\int_0^t \int_\G f(u)\phi w_j\, dS \,d\tau\text{ a.e. }[0,T].
\end{multline}
Here, since $-\Delta_p u_N\to\eta$ weakly in $X'$ we are able to make the identification 
\begin{align*}
\lim_{N\to\infty}\int_0^t \langle -\Delta_p u_N,\phi w_j\rangle_p \,d\tau = \int_0^t \langle \eta , \phi w_j\rangle_p \,d\tau
\end{align*}
in the smaller space $L^{p'}(0,t;(\W)')$.  Now, replacing $\phi(t)$ with $u_{N,j}(t)$ in \eqref{limlaplace-s4}  and summing over $j=1,\ldots,N$ we obtain  
\begin{multline}\label{limlaplace-s5}
\int_0^t \langle \eta,u_N\rangle_p \,d\tau = 	(u'(0),u_N(0))_\Omega - (u'(t),u_N(t))_\Omega \\ 
+ \int_0^t (u',u_N')_\Omega\,d\tau 
-\int_0^t (\grad u',\grad u_N)_\Omega\,d\tau \\ 
-\int_0^t (u',u_N)_\Gamma \,d\tau +\int_0^t \int_\G f(u)u_N \,dS d\tau.
\end{multline}
Taking the limit in \eqref{limlaplace-s5} as $N\to\infty$ we obtain
\begin{multline}\label{limlaplace-s6}
\int_0^t \langle \eta,u\rangle_p \,d\tau = 	(u'(0),u(0))_\Omega - (u'(t),u(t))_\Omega \\ 
+ \int_0^t ||u'(\tau)||_2^2\,d\tau 
-\int_0^t (\grad u',\grad u)_\Omega\,d\tau \\ 
-\int_0^t (u',u)_\Gamma \,d\tau +\int_0^t \int_\G f(u)u\,dS d\tau,
\end{multline}
whose right hand side is identical to the  right hand side of \eqref{limlaplace-s2}  after identifying 
\begin{align*}
\int_0^t (\grad u',\grad u)_\Omega\,d\tau = \frac{1}{2}||\grad u(t)||_2^2 - \frac{1}{2}||\grad u(0)||_2^2
\end{align*} with the aid of \cite[Prop. II.5.11]{Boyer2013}, for instance. 
That is, we have shown
\begin{align*}
\rev\limsup_{N\to\infty}\int_0^t \langle -\Delta_p u_N, u_N\rangle_p\,d\tau \leq \int_0^t \langle \eta,u\rangle_p\,d\tau\text{ a.e. }[0,T].
\end{align*}
Hence,  \eqref{limlaplace-wts} is indeed valid and we have $-\Delta_p u_N\to -\Delta_p u$ weakly in $X'$ completing the proof.
\end{proof}

From Propositions~\ref{prop:limf}, \ref{prop:dampconverg}, and \ref{prop:limlaplace} along with the convergence in \eqref{converg:g} we are now justified in taking the limit in \eqref{limit-1} as $N\to\infty$ and concluding that the limit function $u$ verifies the identity 
\begin{multline}\label{limit-2}
\int_0^t \langle u''(\tau),w_j\rangle_p \,d\tau  
+ \int_0^t \langle -\Delta_p u(\tau), w_j\rangle_p \,d\tau
+ \int_0^t \langle -\Delta_2 u'(\tau),w_j\rangle_2 \,d\tau \\
= \int_0^t \int_\G f(u)w_j \,dS\,d\tau
\end{multline}
for all $j\in\naturals$.  

\subsection{Verification that the limit is a solution}\label{S2.4}
To show that the limit function $u$ does indeed satisfy every criterion of Definition~\ref{def:weaksln} we begin by addressing the identity \eqref{slnid}.

Given any function $\psi\in\W$ we may construct by density a sequence $\{\psi_n\}$ of finite linear combinations of the basis vectors $\{w_j\}$ in the form 
\begin{align*}
\psi_n = \sum_{j=1}^{n} a_{n ,j}w_j
\end{align*}
for a sequence of scalars $\{a_{n,j}\}\subset \R$ such that $\psi_n\to \psi$ strongly in $\W$, exactly as was done in Section~\ref{S2.1}.  Using linearity we may thus replace $w_j$ with $\psi_n$ in \eqref{limit-2}, whereupon taking the limit as $n\to\infty$ we obtain 
\begin{multline}\label{limit-3}
\int_0^t \langle u''(\tau),\psi\rangle_p \,d\tau  
+ \int_0^t \langle -\Delta_p u(\tau), \psi\rangle_p \,d\tau
+ \int_0^t \langle -\Delta_2 u'(\tau),\psi\rangle_2 \,d\tau \\
= \int_0^t \int_\G f(u)\psi \,dS\,d\tau.
\end{multline}
In order to replace $\psi$ with a time dependent function as required in \eqref{slnid} we shall first differentiate \eqref{limit-3} in time, whereupon given any test function $\phi\in C_w([0,T];\W)$ with $\phi_t\in L^2(0,T;\L[2])$ we may replace $\psi$ with $\phi(\tau)$ in (\ref{limit-3}) to obtain 
\begin{multline}\label{limit-4}
\langle u''(\tau),\phi(\tau)\rangle_p   
+ \langle -\Delta_p u(\tau), \phi(\tau)\rangle_p 
+ \langle -\Delta_2 u'(\tau),\phi(\tau)\rangle_2 
= \int_\G f(u(\tau))\phi(\tau) \,dS\ptag{limit-3}
\end{multline}
for a.e. $\tau\in[0,T]$.  With the aid of the product rule in Proposition~\ref{prop:prodrule} (as in \cite[Proposition A.1]{PRT-p-Laplacain}) we may identify the first term in \eqref{limit-4} as  
\begin{align*}
\langle u''(\tau),\phi(\tau)\rangle_p = \frac{d}{dt} (u'(\tau),\phi(\tau))_\O - (u'(\tau),\phi'(\tau))_\O
\end{align*}
for a.e. $\tau\in [0,T]$, whereby we may integrate \eqref{limit-4} on $[0,t]$ to obtain
\begin{multline}\label{limit-5}
(u'(t),\phi(t))_\O - (u_1,\phi(0))_\O - \int_0^t (u'(\tau),\phi'(\tau))_\O\,d\tau  
+ \int_0^t \langle -\Delta_p u(\tau), \phi\rangle_p \,d\tau\\
+ \int_0^t \langle -\Delta_2 u'(\tau),\phi\rangle_2 \,d\tau 
= \int_0^t \int_\G f(u)\phi \,dS\,d\tau\ptag{limit-4}
\end{multline}
Thus, $u$ satisfies the desired identity \eqref{slnid}. In addition, the additional weak continuity in time of $u$ and $u_t$, as required  in items \eqref{def-a} and \eqref{def-b} of Definition \ref{def:weaksln}, is easily addressed by the following proposition:

\begin{prop}\label{prop:uinCw}
The limit function $u$ identified in Corollary~\ref{cor:converg} satisfies both 
	$u\in C_w([0,T];\W)$ and $u_t\in C_w([0,T];\L[2])$.
\end{prop}
\begin{proof}
	As $\W$ is a separable, reflexive Banach space and $\W\into \W[1-\epsilon,p]$ is a continuous mapping, 
	\begin{align*}
	u\in L^\infty(0,T;\W)\cap C([0,T];\W[1-\epsilon,p])\subset C_w([0,T];\W)
	\end{align*}
	by \cite[p. 275]{LM1}, for instance. Thus, the first conclusion of the proposition holds.

	As for the second conclusion, we note that from Corollary \eqref{cor:converg} we see that 
	$u_t\in L^2(0,T;\W[1,2])$ by \eqref{converg:c} and $u_{tt}\in L^2(0,T;(\W)')$ by \eqref{converg:g}.  
	Since the embedding $\rev \W[1,2] \into (\W[1,p])'$ is continuous and $\W[1,2]$ is dense in $\rev(\W[1,p])'$, 
	then up to possible modification on a set of measure zero, we may select a representative $u_t\in C([0,T];(\W[1,p])')$  by \cite[Prop. II.5.11]{Boyer2013}, for instance.  This yields
	\begin{align*}
	u_t\in L^\infty(0,T;\L[2])\cap C([0,T];(\rev\W[1,p])') \subset C_w([0,T];\L[2]),
	\end{align*}
completing the proof.	
 \end{proof}

\subsection{Energy inequality}\label{s1-energy}
 Having demonstrated that $u$ is a solution in the full sense of Definition~\ref{def:weaksln} it now remains only to demonstrate the energy inequality in the statement of Theorem~\ref{thm:exist}.

\begin{prop}\label{energy} The limit function $u$ identified in Corollary~\ref{cor:converg} satisfies the energy inequalities \eqref{energy-2ndid} and \eqref{energy-1stid} in the statement of Theorem~\ref{thm:exist}.
\end{prop}
\begin{proof}
Writing $F(u_N)=\int_0^{u_N}f(s)\,ds$ we have $\frac{d}{dt}F(u_N)=f(u_N)u_N'$ since each $u_N$ is sufficiently regular, so that 
\begin{align*}
	F(u_N(t))-F(u_N(0))=\int_0^t f(u_N)u_N'\,d\tau.
\end{align*}
From \eqref{apriori-3} we see that each $u_N$ satisfies \begin{align}\label{genapriori3}
	\epsilon_N(t) + \int_0^t ||u_N'(\tau)||_{1,2}^2\,d\tau \leq \int_0^t \int_\Gamma f(u_N(\tau))u_N'(\tau)\,dSd\tau + \epsilon_N(0)
\end{align}
with $\epsilon_N(t)=\frac{1}{2}||u_N'(t)||_2^2 + \frac{1}{p}||u_N(t)||_{1,p}^p$, 
so that  
\begin{align}\label{energy1}\rev
	\epsilon_N(t)+\int_0^t ||u_N'(\tau)||_{1,2}^2\,d\tau \leq \int_\Gamma \left( F(u_N(t)) - F(u_N(0))\right)\,dS +\epsilon_N(0).
\end{align}
By defining the energy 
\begin{align*}
E_N(t)=\epsilon_N(t)-\int_{\Gamma} F(u_N(t))\,dS
\end{align*}
we may then re-express \eqref{energy1} as 
\begin{align}\label{energy2}\rev
E_N(t)+\int_0^t ||u_N'(\tau)||_{1,2}^2 \,d\tau \leq E_N(0)\ptag{energy1}
\end{align}
A first step towards the desired inequality is the convergence of the terms $\int_{\Gamma}F(\cdot)\,dS$.

 From the mean value theorem and Remark~\ref{rmk:fbound} we know that 
\begin{align*}
	|F(u_N)-F(u)| = |f(\xi)||u_N-u| \leq C(1+|\xi|^r)|u_N-u|
\end{align*}
for some $\xi$ with $|\xi|\leq |u_N|+|u|$. Thus,
\begin{align}\label{energy-2.5}
\int_\Gamma |F(u_N)-F(u)|\,dS 
& \leq C\int_\Gamma (1+|u_N|^r + |u|^r)|u_N-u|\,dS\notag \\
& \leq  C(1+|u_N|_{4r/3}^r + |u|_{4r/3}^r)|u_N-u|_4\notag \\
& \leq C(1+||u_N||_{1-\epsilon,p}^r + ||u||_{1-\epsilon,p}^r)||u_N-u||_{1-\epsilon,p}
\end{align}
by H\"older's inequality and the trace mappings $\W[1-\epsilon,p]\tinto \Lb[4r/3]$ and $\W[1-\epsilon,p]\tinto\Lb[4]$. Since for any $t\in[0,T]$ we have $||u_N||_{1-\epsilon,p}$ and $||u||_{1-\epsilon,p}$ bounded and $||u_N-u||_{1-\epsilon,p}\to 0$ by virtue of \eqref{converg:d}, it follows that

\begin{align}\label{energy-F}
\lim_{N\to\infty}\int_\Gamma F(u_N(t))\,dS = \int_\Gamma F(u(t))\,dS;\quad t\in[0,T].
\end{align}
We may thus appeal to the lower semi-continuity of norms to conclude from \eqref{energy2} and \eqref{energy-F} that 
\begin{align*}
E(t) + \int_0^t ||u'(\tau)||_{1,2}^2\,d\tau &\leq 
\liminf_{N\to\infty }\left( E_N(t) + \int_0^t  ||u_N'(\tau)||_{1,2}^2\,d\tau \right) \notag\\
&\rev\leq  \liminf_{N\to\infty} E_N(0)
\end{align*}
with $E$ as in the statement of Theorem~\ref{thm:exist}. However, from the convergence in Proposition~\ref{prop:ICs} and \eqref{energy-F} at $t=0$ we obtain $\lim_{N\to\infty}E_N(0)=E(0)$ from which the desired energy inequality \eqref{energy-1stid} follows.

Finally we verify that, in fact,  
\begin{align}\label{energy-wts2}
\int_\Gamma F(u(t))\,dS-\int_\Gamma F(u(0))\,dS=\int_0^t\int_\Gamma f(u(\tau))u'(\tau)\,dSd\tau
\end{align} so that $u$ additionally satisfies the energy inequality \eqref{energy-2ndid} in Theorem~\ref{thm:exist}.  In considering \eqref{energy-F} we may demonstrate \eqref{energy-wts2} by equivalently showing that  
\begin{align*}
\lim_{N\to\infty} \int_0^t \int_\Gamma f(u_N(\tau))u_N'(\tau)\,dS d\tau = \int_0^t \int_\Gamma f(u(\tau))u'(\tau)\,dS d\tau.
\end{align*}
To see this, we first estimate  
\begin{multline}\label{energy-3}
\int_0^t \int_\Gamma \left|f(u_N(\tau))u_N'(\tau) - f(u(\tau))u'(\tau)\right|\,dSd\tau \\ \leq \underbrace{\int_0^t\int_\Gamma|f(u_N(\tau))||u_N'(\tau)-u'(\tau)|\,dSd\tau}_\text{(i)} \\
+  \underbrace{\int_0^t\int_\Gamma|u'(\tau)||f(u_N(\tau))-f(u(\tau))|\,dSd\tau}_\text{(ii)}.
\end{multline}
The first term in \eqref{energy-3} may controlled by H\"older's inequality with conjugate exponents $\alpha =4/(1+2\epsilon)$ and $\alpha'=4/(3-2\epsilon)$ yielding 
\begin{align*}
\text{(i)}\leq t\sup_{\tau\in[0,t]}|f(u_N(\tau))|_{\alpha'}||u_N'(\tau)-u'(\tau)||_{1-\epsilon,2} 
\end{align*}
from the trace mapping $\W[1-\epsilon,2]\tinto \Lb[\alpha]$. Since $\alpha'$ decreases to $4/3$ as $\epsilon$ decreases to zero we may select $\epsilon>0$ sufficiently small so that $\alpha'<3/2$. From the restrictions on $r$ given in \eqref{ass:f} we thus obtain $r\alpha'<2p/(3-p)$ and hence $|f(u_N)|_{\alpha'}\leq C(1+|u_N|_{r\alpha'}^r)$ is bounded on $[0,t]$ by \eqref{apriori-a}. Since $||u_N'(\tau)-u'(\tau)||_{1-\epsilon,2}\to 0$ for all $\tau\in[0,t]$ by \eqref{converg:d} we thus obtain $\text{(i)}\to 0$ as $N\to\infty$.

For the second term in \eqref{energy-3} H\"older's inequality and Lemma~\ref{lem:f-lipshitz} yield 
\begin{align*}
\text{(ii)} &\leq \int_0^t |u'(\tau)|_4 |f(u_N(\tau)) - f(u(\tau))|_{4/3}\,d\tau\\
&\leq C\sup_{\tau\in[0,t]}||u_N(\tau)-u(\tau)||_{1-\epsilon,p}||u'||_{L^1(0,t;\Lb[4])},
\end{align*}
whereby $\text{(ii)}\to 0$ as $N\to\infty$ from \eqref{converg:d} and the fact that $u'\in L^2(0,T;\W[1,2])\into L^1(0,t;\Lb[4])$ from \eqref{converg:c}. This establishes \eqref{energy-wts2} and completes
the proof.
\end{proof}

\section{Solutions for other sources}\label{S3}
In order to extend the results from Section~\ref{S2} to more general sources we shall employ standard truncation arguments similar to \cite{BL1, CEL1, PRT-p-Laplacain, RW}. As in these papers, a key step is obtaining solutions in the case where $f:\W\to\Lb[2]$ is locally Lipschitz with care taken to bound the interval of existence independent from this particular Lipschitz constant. 
\subsection{Locally Lipschitz sources}\label{S3.1}
\begin{prop}\label{prop:ll}
	Assume that $f:\W\to\Lb[2]$ is locally Lipschitz continuous. From the initial data $(u_0,u_1)$ we may find a constant $K$ so that problem \eqref{wave} possesses a local solution in the sense of Definition~\ref{def:weaksln} on an interval $[0,T]$ whose length is dependent only upon $K$ and the local Lipschitz constant of $f:\W\to\Lb[4/3]$ on the ball of radius $K$ about zero in $\W$. Further, $||u(t)||_{1,p}\leq K$ on $[0,T]$. 
\end{prop}
\begin{proof}
	For a constant $K>0$ define 
	\begin{align*}
	f_K(u)=\begin{cases}
	f(u),&\text{ for }||u||_{1,p} \leq K\\
	f\left(\frac{Ku}{||u||_{1,p}}\right),&\text{ for }|| u||_{1,p} >K
	\end{cases}
	\end{align*}
	and the corresponding $K$ problem
	\begin{align}
	\label{ll-k}
	\begin{cases}
	u_{tt}-\Delta_p u -\Delta u_t = 0 &\text{ in } \Omega \times (0,T),\\[.1in]
	(u(0),u_t(0))=(u_0,u_1),\\[.1in]
	|\grad u|^{p-2}\partial_\nu u + |u|^{p-2}u + \partial_\nu u_t + u_t = f_K(u)&\text{ on }\Gamma \times(0,T).
	\end{cases}
	\end{align}
	While it is readily verified that each $f_K$ is globally Lipschitz continuous as a map from $\W$ into $\Lb[2]$ a proof is provided in Lemma~\ref{lem:ll-trunc} in the Appendix, as most references in the literature (e.g. \cite{CEL1,GR,RW}) assume a Hilbert space structure on the domain not available here. Following from this, the results from Section~\ref{S2} provide a weak solution $u_K$ of each problem \eqref{ll-k} on an interval $[0,T]$ satisfying the energy identity \eqref{energy-2ndid}. (In fact, each $u_K$ exists on $[0,\infty)$ as $T$ is arbitrary for the globally Lipschitz source $f_K$.) Since $f_K:\W \to\Lb[4/3]$ is likewise globally Lipschitz continuous, we may estimate with H\"older and Young's inequalities
	\begin{align*}
		\int_\Gamma f_K(u_K)u_K'\,dS &\leq C_\epsilon  (1+|f_K(u_K)|_{4/3})^2 + \epsilon |u_K'|_4^2\\
		&\leq C_\epsilon (1+||u_K||_{1,p})^2 + C\epsilon||u_K'||_{1,2}^2\\
		&\leq C_K(1+||u_K||_{1,p}^p) + \frac{1}{2}||u_K'||_{1,2}^2.
	\end{align*}
	for a suitable choice of $\epsilon$ and a constant $C_K$ dependent only upon the Lipschitz constant of $f$ as a map into $\Lb[4/3]$. That is, 
	\begin{align}\label{ll-2}
	\int_0^t \int_{\Gamma} f_K(u_K)u_K'\,dSd\tau \leq C_K \int_0^t (1+p\E(\tau))\,d\tau + \frac{1}{2}\int_0^t ||u_K'||_{1,2}^2\,d\tau.
	\end{align}
	In applying \eqref{ll-2} to the energy inequality \eqref{energy-2ndid} we thus obtain 
	\begin{align}
	\E(t) + \frac{1}{2}\int_0^t ||u_K'(\tau)||_{1,2}^2 \leq \E(0) + C_K\int_0^t (1+p\E(\tau))\,d\tau.
	\end{align}
	Gronwall and Young's inequalities now imply that 
	\begin{align*}
		2\E(t)\leq 2(\E(0)+C_Kt)\exp(pC_kt)\leq (\E(0)+C_kt)^2 + \exp(2pC_kt)
	\end{align*} whereupon selecting a fixed $K$ with $K^{p/2}>\sqrt{2p}\E(0)$,  $K^p >2p$,  and 
	\begin{align*}\rev
	T_0=\min\left\{\frac{1}{\sqrt{2p}C_K}(K^{p/2}-\sqrt{2p}\E(0)),\, \frac{1}{2pC_K}(\ln(K^p)-\ln(2p))  \right\}
	\end{align*}
	we obtain $\E(t)\leq \frac{K^p}{p}$ on $[0,T_0]$.  As such, 
	\begin{align}
	||u_k||_{1,p}\leq \left(p\E(t)\right)^{1/p}\leq K
	\end{align}
	on $[0,T_0]$ so that $f_K=f$ on this interval.  The solution $u_K$ of \eqref{ll-k} is therefore, in fact, a solution of the original problem \eqref{wave} in the full sense of Definition~\ref{def:weaksln} on  $[0,T_0]$ as prescribed in the statement of the proposition. Moreover, this solution verifies the energy inequalities  detailed in Section~\ref{s1-energy}.
\end{proof}

\subsection{General sources}\label{S3.2}
In order to establish the existence of solutions for more general sources satisfying \eqref{ass:f} we employ another truncation argument as in \cite{PRT-p-Laplacain, RW}.  To begin, select as in \cite{Radu1} a sequence $\{\eta_n\}\subset C^\infty(\R)$ of cutoff functions such that 
\begin{align*}
0\leq\eta_n\leq 1,\quad |\eta_n'(s)|\leq \frac{C}{n},\quad\text{and }
\begin{cases}
 \eta_n(s)=1,&\text{ for }|s|\leq n,\\
 \eta_n(s)=0,&\text{ for }|s|>2n
\end{cases}
\end{align*}
for some constant $C$ independent from $n$ and define 
\begin{align}\label{sc-f}
f_n(u)=f(u)\eta_n(u).
\end{align}
In order to leverage the results of Section~\ref{S3.1} it must be shown that each $f_n$ is locally Lipschitz continuous as a map into $\Lb[2]$ and exhibit a bound on the magnitude of the Lipschitz constants of each $f_n$ as a map into $\Lb[4/3]$ which is uniform in $n$.  

\begin{lem}\label{lem:sc-lip} Each $f_n$ given by \eqref{sc-f} satisfies
	\begin{enumerate}[(i)]
		\item\label{sc-lip-1} $f_n:\W \to \Lb[2]$ is globally Lipschitz continuous;
		\item\label{sc-lip-2} $f_n:\W[1-\epsilon,p] \to \Lb[4/3]$ is locally Lipschitz continuous with a Lipschitz  constant \underbar{independent} of $n$ on any ball of radius $R$ about $0$ in $\W[1-\epsilon,p]$.
	\end{enumerate}
\end{lem}

The proof of this lemma is very similar in structure to \cite[Lem. 2.4]{RW} but sufficiently different due to the boundary terms to merit a proof which is presented in the Appendix.

  In light of the previous lemma we are now able to utilize the results of Section~\ref{S3.1} to produce a sequence $\{u^n\}$ of solutions to the approximated $n$th problem 	
  \begin{align}
  \label{sc-n}
  \begin{cases}
  u^n_{tt}-\Delta_p u^n -\Delta u^n_t = 0 &\text{ in } \Omega \times (0,T),\\[.1in]
  (u^n(0),u^n_t(0))=(u_0,u_1),\\[.1in]
  |\grad u^n|^{p-2}\partial_\nu u^n + |u^n|^{p-2}u^n + \partial_\nu u^n_t + u^n_t = f_n(u^n)&\text{ on }\Gamma \times(0,T).
  \end{cases}
  \end{align}
  in the sense of Definition~\ref{def:weaksln} on a non-degenerate interval $[0,T]$. (This can be done precisely because Lemma~\ref{lem:sc-lip} asserts that each $f_n$ given by \eqref{sc-f} satisfies the requirements of Proposition~\ref{prop:ll} with local Lipschitz constants as maps into $\Lb[4/3]$ independent of $n$.)  Further, since each $u^n$ satisfies the energy inequality 
  \begin{align}\label{generalized-apriori-3}
   \E^n(t)+\int_0^t ||u^n_t(\tau)||_{1,2}^2\,d\tau \leq \E^n(0) + \int_0^t\int_{\Gamma} f^n(u^n)u^n_t\,dSd\tau
  \end{align}
  with 
  \begin{align*}
  \E^n(t)=\frac{1}{2}||u_t^n(t)||_2^2 + \frac{1}{p}||u^n(t)||_{1,p}^p
  \end{align*}
   we may produce as in the a priori estimates in Proposition~\ref{prop:apriori} the bound $\E^n(t)\leq C$ on $[0,T]$ by virtue of the fact that $||u(t)||_{1,p}\leq K$ on $[0,T]$ from Proposition~\ref{prop:ll}.   Correspondingly, we conclude that 
   \begin{subequations}
   \begin{align}
   \{u^n\}\text{ is a bounded sequence in }&L^\infty(0,T;\W), \\
   \{u_t^n\}\text{ is a bounded sequence in }&L^\infty(0,T;\L[2]), \\
   \{u_t^n\}\text{ is a bounded sequence in }&L^2(0,T;\W[1,2]). 
   \end{align}
   Further, since each $u^n$ must in particular satisfy \eqref{slnid} we have at each $\phi\in\W$ that 
   \begin{align*}
   |\langle u^n_{tt}(t),\phi\rangle_p| &= \left|\frac{d}{dt}(u^n_t(t),\phi)_\Omega\right|\\
   &\leq |\langle -\Delta_p u^n(t),\phi\rangle_p| + |(\grad u^n_t(t),\grad\phi)_\Omega| + |(u^n_t(t),\phi)_\Gamma| + |(f_n(u^n(t)),\phi)_\Gamma|\\
   &\leq \left( ||u^n(t)||_{1,p}^{p-1} + ||u_t^n(t)||_{1,p'}\right)||\phi||_{1,p} + |f_n(u^n(t))|_{4/3}|\phi|_{4}   
	\end{align*}
   by H\"older's inequality with $p'$ conjugate to $p$. Since $p'<2$ and $|f_n(u^n(t))|_{4/3}\leq C(1+||u^n(t)||_{1,p})$ given that $f^n$ is locally Lipschitz as in Lemma~\ref{lem:sc-lip} Item \eqref{sc-lip-2} we may thus additionally conclude that 
   \begin{align}
   \{u^n_{tt}\}\text{ is a bounded sequence in }L^2(0,T,(\W)').
   \end{align}
   \end{subequations}
   As was the case in Corollary~\ref{cor:converg}, the standard compactness results now imply the following:
   \begin{cor}\label{cor:sc-converg} Up to a subsequence, the sequence of solutions $\{u^n\}$ of \eqref{sc-n} satisfies 
  \begin{subequations}
  \begin{alignat}{2}
                 u^n&\to u&\text{ weak* in }&L^\infty(0,T;\W),\label{sc-converg-a}\\
	   	u^n_t &\to u_t&\text{ weak* in }&L^\infty(0,T;\L[2]),\label{sc-converg-b}\\
	   	u^n_t &\to u_t&\text{ weakly in }&L^2(0,T;\W[1,2])\label{sc-converg-c}\\
	   	u^n &\to u&\text{ strongly in }&C([0,T];\W[1-\epsilon,p]),\label{sc-converg-d}\\
	   	u^n_t &\to u_t&\text{ strongly in }&L^2(0,T;\W[1-\epsilon,2]),\label{sc-converg-e}\\
	  	u^n_{tt} &\to u_{tt}&\text{ weakly in }&L^2(0,T;(\W[1,p])').\label{sc-converg-f}  
   \end{alignat}
	\end{subequations}
   \end{cor}
   
 Following in line with the commentary at the beginning of Section~\ref{S2.4} we shall now demonstrate that sequence $\{u^n\}$ and the limit function $u$ identified in Corollary~\ref{cor:sc-converg} are a solution to \eqref{wave} in the sense of Definition~\ref{def:weaksln} by utilizing many of the same arguments.  In order to do so we must establish an analogue of Proposition~\ref{prop:limf} and show that the results of Propositions~\ref{prop:ICs}, \ref{prop:dampconverg}, and \ref{prop:limlaplace} apply. 
   	
   One obvious difference between the solutions $\{u^n\}$ obtained here is that they verify the identity 
   \begin{multline}\label{sc-slnid}
   	(u^n_t(t),\phi(t))_\O - (u_1,\phi(0))_\O - \int_0^t (u^n_t(\tau),\phi_t(\tau))_\O\,d\tau 
   	+ \int_0^t \langle -\Delta_p u^n(\tau),\phi(\tau)\rangle_p\,d\tau \\
   	+ \int_0^t \langle -\Delta_2 u^n_t(\tau),\phi(\tau)\rangle_2\,d\tau 
   	= \int_0^t \int_\G f_n(u^n(\tau))\phi(\tau)\,dSd\tau 
   \end{multline}
for all test functions $\phi\in C_w([0,T];\W)$ with $\phi_t\in L^2(0,T;\W[1,2])$ in lieu of a fixed source $f$.  The following proposition (an analogue of Proposition~\ref{prop:limf}) assures that this key difference is of no practical consequence. 

\begin{prop}[c.f. Prop.~\ref{prop:limf}]\label{prop:sc-f} The sequence of solutions of problem \eqref{sc-f} satisfies
	\begin{align*}
	f_n(u^n)\to f(u)\text{ strongly in }L^\infty(0,T;\Lb[4/3])
	\end{align*}
\end{prop}
\begin{proof}
From the definition of $f_n$ in (\ref{sc-f}) we have
\begin{align}\label{sc-f-1}
|f_n(u^n(t)) - f(u(t))|_{4/3} &\leq \underbrace{|f_n(u^n(t))-f_n(u(t))|_{4/3}}_\text{(i)} + \underbrace{|f_n(u(t))-f(u(t))|_{4/3}}_\text{(ii)}.
\end{align}
For $\text{(i)}$, the sequence $\{||u^n(t)||_{1-\epsilon,p}\}$ is bounded (say by $R$) uniformly on $[0,T]$ by virtue of \eqref{sc-converg-d}, so that by the local Lipschitz continuity of $f_n$ in Lemma~\ref{lem:sc-lip} we have $\text{(i)}\leq C_R ||u_n(t)-u(t)||_{1-\epsilon,p}$ for all $t\in[0,T]$ with the Lipschitz constant $C_R$ independent of $n$.  Taking the limit as $n\to\infty$ we thus obtain from \eqref{sc-converg-d} that
\begin{align}\label{sc-f-2}
\text{(i)}\leq C_R ||u_n(t)-u(t)||_{1-\epsilon,p}\to 0
\end{align}
for all $t\in[0,T]$.
For $\text{(ii)}$ it is clear that $f_n(u(t))\to f(u(t))$ pointwise almost everywhere on $\G$, so that from the bound
\begin{align*}
|f_n(u(t))-f(u(t))| = |\eta_n(u(t)))-1||f(u(t))|\leq 2|f(u(t))|
\end{align*}
with $|f(u(t))|_{4/3}\leq C(||u(t)||_{1-\epsilon,p}+1)$ by Lemma~\ref{lem:f-lipshitz}. Thus,  $|f(u(t))|_{4/3}<\infty$, and by the Lebesgue dominated convergence theorem  we obtain 
\begin{align}\label{sc-f-3}
|f_n(u(t))-f(u(t))|_{4/3}\to 0.
\end{align}
The result now follows by applying \eqref{sc-f-2} and \eqref{sc-f-3} to \eqref{sc-f-1}.
\end{proof}

Since each approximate solution $u^n$ satisfies $(u^n(0),u^n_t(0))=(u_0,u_1)$ in $\W\times\L[2]$ it is clear that the results of Proposition~\ref{prop:ICs} still apply, and similarly we obtain $-\Delta_2 u^n_t \to -\Delta_2 u_t$ weakly in $L^2(0,T;(\W)')$ precisely as in Proposition~\ref{prop:dampconverg}.  We thus turn our attention next to verifying the convergence of the terms due to the $p$-Laplacian in line with Proposition~\ref{prop:limlaplace} using an extremely similar argument.  
\begin{prop}[c.f Prop.~\ref{prop:limlaplace}] Up to a subsequence, the sequence of solutions $\{u^n\}$ of \eqref{sc-n} and the limit function $u$ in Corollary~\ref{cor:sc-converg} satisfy 
		\begin{align*}
		-\Delta_pu^n \to -\Delta_p u\text{ weak* in }L^\infty(0,T;(\W)').
		\end{align*}
\end{prop}
\begin{proof}
	We shall inherit the framework of the proof of Proposition~\ref{prop:limlaplace} by taking $X=L^p(0,T;\W)$ and the $p$-Laplacian extended to a maximal monotone operator $-\Delta_p:X\to X'$.  As in that proof, the bounds from \eqref{sc-converg-a} can be used to show that there exists some $\eta\in X'$ so that 
	\begin{align*}
	-\Delta_p u^n\to\eta\text{ weakly in }X',
	\end{align*}
	and likewise we may conclude thanks to the properties of maximal monotone operators that $\eta=-\Delta_p u$ in $X'$ provided 
	\begin{align}\label{sc-limlaplace-wts}
	\limsup_{n\to\infty} \langle -\Delta_p u^n,u^n\rangle_{X',X} \leq \langle \eta, u\rangle_{X',X}.
	\end{align}

	By taking $\phi=u^n$ in \eqref{sc-slnid} and rearranging we obtain
	\begin{multline}\label{sc-laplace1}
\int_0^t \langle -\Delta_p u^n,u^n\rangle_p\,d\tau =
	   	 (u_1,u^n(0))_\O -(u^n_t(t),u^n(t))_\O + \int_0^t ||u^n_t(\tau)||_2^2\,d\tau 	\\
	   	- \int_0^t \langle -\Delta_2 u^n_t(\tau),u^n(\tau)\rangle_2\,d\tau 
	   	+ \int_0^t \int_\G f_n(u^n(\tau))u^n(\tau)\,dSd\tau 
	\end{multline}	
	that upon integration by parts,
	\begin{multline}\label{sc-laplace2}
		\int_0^t \langle -\Delta_p u^n,u^n\rangle_p\,d\tau =
		\underbrace{  (u_1,u^n(0))_\O -(u^n_t(t),u^n(t))_\O}_\text{(i)} \\
		+ \underbrace{\int_0^t ||u^n_t(\tau)||_2^2\,d\tau}_\text{(ii)}
		- \underbrace{\frac{1}{2}||\grad u^n(t)||_2^2 + \frac{1}{2}||\grad u^n(0)||_2^2}_\text{(iii)}\\ 
		-\underbrace{\int_0^t (u^n_t,u^n)_\G\,d\tau}_\text{(iv)}
		+ \underbrace{\int_0^t \int_\G f_n(u^n(\tau))u^n(\tau)\,dSd\tau}_\text{(v)} .
	\end{multline}
	Since this expression is identical to \eqref{limlaplace-s1pp} in Proposition~\ref{prop:limlaplace} we may justify taking the limit superior in \eqref{sc-laplace2} using the exact same arguments on each of the terms $\text{(i)}$ through $\text{(v)}$. Thus,
	\begin{multline}\label{sc-laplace3}
		\limsup_{n\to\infty}\int_0^t \int_0^t \langle -\Delta_p u^n,u^n\rangle_p\,d\tau \leq (u_1,u(0))_\O - (u_t(t),u(t))_\O \\
		+ \int_0^t ||u_t(\tau)||_2^2\,d\tau 
		-\frac{1}{2}||\grad u(t)||_2^2 + \frac{1}{2}||\grad u(0)||_2^2\\
		-\int_0^t(u_t,u)_\G\,d\tau + \int_0^t \int_\G f(u)u\,dSd\tau\text{ a.e. }[0,T].
	\end{multline}
	Again, we shall attempt to express the right hand side of \eqref{sc-laplace3} through a different means.  By taking $\phi=u$ in \eqref{sc-slnid}, 
	\begin{multline}\label{sc-laplace4}
		\int_0^t \langle -\Delta_p u^n,u\rangle_p\,d\tau =
		(u_1,u(0))_\O -(u^n_t(t),u(t))_\O + \int_0^t (u^n_t(\tau),u_t(\tau))_\O\,d\tau 	\\
		- \int_0^t \langle -\Delta_2 u^n_t(\tau),u(\tau)\rangle_2\,d\tau 
		+ \int_0^t \int_\G f_n(u^n(\tau))u(\tau)\,dSd\tau.
	\end{multline}	
	Taking the limit as $n\to\infty$ in \eqref{sc-laplace4} is readily justified, so that 
	\begin{multline}\label{sc-laplace5}
		\int_0^t \langle \eta,u\rangle_p\,d\tau =  \lim_{n\to\infty}\int_0^t \langle -\Delta_p u^n,u\rangle_p\,d\tau = (u_1,u(0))_\O - (u_t(t),u(t))_\O \\
		+ \int_0^t ||u_t(\tau)||_2^2\,d\tau 
		-\frac{1}{2}||\grad u(t)||_2^2 + \frac{1}{2}||\grad u(0)||_2^2\\
		-\int_0^t(u_t,u)_\G\,d\tau + \int_0^t \int_\G f(u)u\,dSd\tau\text{ a.e. }[0,T].
	\end{multline}
	\sloppy Combining \eqref{sc-laplace3} and \eqref{sc-laplace5} we achieve the desired inequality in \eqref{sc-limlaplace-wts}. The conclusion then follows immediately as the sequence $\{-\Delta_p u^n\}$ is bounded in $L^\infty(0,T;(\W)')$ from \eqref{sc-converg-a} and following verbatim the same calculation as in the onset of Proposition~\ref{prop:limlaplace}.
\end{proof}

With this result established, it is seen that the limit function $u$ indeed verifies the identity \eqref{slnid} in Definition~\ref{def:weaksln}.  Moreover, it is easily verified that $u\in C_w([0,T];\W)$ and that $u_t\in C_w([0,T];\L[2])$ using precisely the same arguments as in Proposition~\ref{prop:uinCw}.  Thus, it remains only to show that the energy inequalities still apply in line with Proposition~\ref{energy}.

\begin{prop}[c.f. Prop.~\ref{energy}]\label{sc-energy} The limit function $u$ identified in Corollary~\ref{cor:sc-converg} satisfies the energy inequalities \eqref{energy-2ndid} and \eqref{energy-1stid} in the statement of Theorem~\ref{thm:exist}.
\end{prop}
\begin{proof}The proof here is essentially unchanged from Proposition~\ref{energy}. Since each approximate solution $u^n$ verifies \eqref{energy-1stid} we obtain 
\begin{align}\label{sc-energy1}
E_n(t) + \int_0^t ||u^n_t(\tau)||_{1,2}^2\,d\tau \leq E_n(0)
\end{align}
with $E_n(t)=\frac{1}{2}||u^n_t(t)||_2^2 + \frac{1}{p}||u^n(t)||_{1,p}^p - \int_\G F_n(u^n(t))\,dS$ by taking 
$F_n(u^n)=\int_0^{u^n}f_n(s)\,ds$. From the mean value theorem,
\begin{align*}
|F_n(u^n) - F_n(0)| &= |f_n(\xi)||u^n-u|\\
		&\leq |f(\xi)||u^n-u|
\end{align*}
for some $\xi$ with $|\xi|\leq |u^n|+|u|$. Thus,
\begin{align*}
\int_\G |F_n(u^n)-F_n(u)|\,dS \leq C(1+||u^n|||_{1-\epsilon,p}^r + ||u||_{1-\epsilon,p}^r)||u^n-u||_{1-\epsilon,p}
\end{align*}
by the same calculation as in \eqref{energy-2.5}, whereupon it follows from \eqref{sc-converg-d} that 
\begin{align*}
\lim_{n\to\infty} \int_\G F_n(u^n(t))\,dS = \int_G F(u(t))\,dS;\quad t\in[0,T].
\end{align*}
Here as in Proposition~\ref{energy}, lower-semicontinuity of the norms along with \eqref{sc-energy1} yields the desired result of 
\begin{align*}
E(t)+\int_0^t ||u'(\tau)||_{1,2}^2\,d\tau \leq E(0)
\end{align*}
with $E$ as in the statement of Theorem~\ref{thm:exist} upon noting that $E_n(0)=E(0)$ for all $n$ as $(u^n(0),u^n_t(0))=(u_0,u_1)$ in $\W\times\L[2]$ for all $n$. 

To verify \eqref{energy-1stid} it is sufficient to demonstrate that 
\begin{align*}
\lim_{n\to\infty}\int_0^t \int_\G f_n (u^n(\tau))u^n_t(\tau)\,d\tau dS = \int_0^t \int_\G f(u(\tau))u_t(\tau)\,d\tau dS
\end{align*}
which is accomplished through the estimate 
\begin{multline}\label{sc-energy-3}
\int_0^t \int_\G |f_n(u^n(\tau))u^n_t(\tau) - f(u(\tau))u_t(\tau)|\,dSd\tau\\
\leq 
\underbrace{\int_0^t \int_\G |f_n(u^n(\tau))||u^n_t(\tau) - u_t(\tau)|\,dSd\tau}_\text{(i)} \\
+ 
\underbrace{\int_0^t \int_\G |u_t||f_n(u^n)-f_n(u)|\,dSd\tau}_\text{(ii)}
 +\underbrace{ \int_0^t\int_\G |u_t||f_n(u)-f(u)|\,dSd\tau}_\text{(iii)}.
\end{multline}
analogously to \eqref{energy-3} in Proposition~\ref{energy}. Since 
\begin{align*}
\text{(i)}\leq \int_0^t\int_\G |f(u^n(\tau))||u^n_t(\tau)-u_t(\tau)|\,dSd\tau
\end{align*}
we obtain $\text{(i)}\to 0$ through the exact argument used in the aforementioned proposition. For the second term in \eqref{sc-energy-3} we know that each $f_n$ is locally Lipschitz with constant not dependent upon $n$ as a map from $\W[1-\epsilon,p]$ into $\Lb[4/3]$ from Lemma~\ref{lem:sc-lip}, and thus 
\begin{align*}
\text{(ii)}\leq C\sup_{\tau\in[0,t]} ||u^n(\tau)-u(\tau)||_{1-\epsilon,p}||u_t||_{L^1(0,t;\Lb[4])} \to 0
\end{align*}
exactly as in item $\text{(ii)}$ of the Proposition~\ref{energy}.
Finally, since $f_n\to f$ pointwise a.e. and $|u_t||f_n(u)-f(u)|\leq 2|u_t||f(u)|\in L^1(0,T;\Lb[1])$ we achieve $\text{(iii)}\to 0$ by the dominated convergence theorem.
\end{proof}
This completes the proof of Theorem~\ref{thm:exist}.

\section{Global Solutions}\label{S4}
It has been shown in Section~\ref{S2} that global solutions of \eqref{wave} exist in the case where $f:\W \to \Lb[2]$ is globally Lipschitz continuous.  As this condition is assured only by taking $r=1$ (corresponding to an essentially linear source term), we seek a more meaningful bound on $r$ assuring global solutions.

As in \cite{GR, PRT-p-Laplacain} it is the case here that either the solution $u$ must, in fact, be global in time or else one may find a value of $T_0$ with $0<T_0<\infty$ so that
\begin{align}\label{global-blowup}
\limsup_{t\to T_0^-}\left( \mathscr E(t) + \int_0^t ||u_t(\tau)||_{1,2}^2\,d\tau\right) =\infty
\end{align}
with $\mathscr{E}(t)=\frac{1}{2} ||u_t(t)||_2^2 + \frac{1}{p}||u(t)||_{1,p}^p$ as in \eqref{energy-2ndid} from Theorem~\ref{thm:exist}.  In order to show this we demonstrate a bound on the energy
\begin{align}
\E(t) + \int_0^t ||u_t(\tau)||_{1,2}^2\,d\tau
\end{align}
on $[0,T]$ which is dependent only upon $T$ and the energy of the initial data, $\E(0)$.  With this bound, the situation described in \eqref{global-blowup} clearly cannot apply since one could always bound this quantity away from infinity on any finite interval.

The following proposition thus establishes the desired result.

\begin{prop}
	If $u$ is a weak solution of \eqref{wave} given by Theorem~\ref{thm:exist} on $[0,T]$ and $r\leq p/2$, then there exists a constant $M$ dependent upon $T$ and $\E(0)$ so that
	\begin{align*}
	 \E(t) + \int_0^t ||u_t(\tau)||_{1,2}^2\,d\tau < M\quad t\in[0,T].
	\end{align*}
\end{prop}
\begin{proof}
	Since $u$ satisfies the energy identity
	\begin{align}\label{global-energyid}
	\E(t) + \int_0^t ||u_t(\tau)||_{1,2}^2\,d\tau \leq \E(0) + \int_0^t \int_\Gamma f(u(\tau))u_t(\tau)\,dSd\tau
	\end{align}
	we may bound $\E(t)$ using Gronwall's inequality as in Proposition~\ref{prop:apriori} provided the source term can be adequately controlled.  Recalling that $|f(u)|\leq C(|u|^r+1)$ from Remark~\ref{rmk:fbound} we may estimate from H\"older and Young's inequality with $\epsilon$ that
	\begin{align}\label{gb-1a}
	\int_\G f(u(\tau))u_t(\tau)\,dS &\leq |f(u(\tau))|_2|u_t(\tau)|_2\notag \\
	&\leq C_\epsilon(1+|u|_{2r}^{2r}) + \frac{1}{2}||u_t(\tau)||_{1,2}^2
	\end{align}
	for a suitable choice of $\epsilon$.   By noting that
	$$|a|^{2r}\leq (1+|a|)^{2r}\leq (1+|a|)^p\leq C(1+|a|^p)$$
	for $a\in\R$ and that $2r\leq p < 2p/(3-p)$ we
	may use continuity of the map $\W\tinto\Lb[2r]$ to obtain
	\begin{align}\label{gb-1}
	\int_\G f(u(\tau))u_t(\tau)\,dS\leq C(1+||u(\tau)||_{1,p}^p) + \frac{1}{2}||u_t(\tau)||_{1,2}^2\ptag{gb-1a}
	\end{align}
	from \eqref{gb-1a}.
	Integrating \eqref{gb-1} on $[0,t]$ we see that
	\begin{align*}
	\int_0^t\int_\G f(u(\tau))u_t(\tau)\,dSd\tau \leq C\int_0^t (1+\E(\tau))\,d\tau + \frac{1}{2}\int_0^t ||u_t(\tau)||_{1,2}^2\,d\tau
	\end{align*}
	as $||u(\tau)||_{1,p}^p \leq p\E(\tau)$, and applying this bound to \eqref{global-energyid} we obtain
	\begin{align}\label{gb-2}
	\E(t)+\frac{1}{2}\int_0^t ||u_t(\tau)||_{1,2}^2\,d\tau \leq \E(0) + C\int_0^t (1+\E(\tau))\,d\tau.
	\end{align}
	By Gronwall's inequality, \eqref{gb-2} now implies that
	\begin{align*}
	\E(t)\leq (\E(0)+Ct)\exp(Ct)< N
	\end{align*}
	for a constant $N$ on $[0,T]$, which in turn yields
	\begin{align}
	\E(t)+\int_0^t ||u_t(\tau)||_{1,2}^2 \leq \underbrace{2\E(0)+2CT(1+N)}_M, \quad t\in [0,T].
	\end{align}	
 The desired result then follows by taking $M$ to be the indicated constant $2\E(0)+2CT(1+N)$.
	\end{proof}

\noindent\textbf{Acknowledgements: } The authors wish to express their gratitude to the anonymous referees whose valuable feedback has served to substantively improve the quality of this manuscript. We are also thankful to Professor Hideo Kubo for communicating the manuscript.

\section{Appendix}\label{A}

{ \rev\begin{lem}[The $p$-Laplacian is maximal monotone on $X$]\label{lem:plaplace-mmono}
	The $p$-Laplacian $-\Delta_p:X \to X'$  given by 
	\begin{align*}
	\langle -\Delta_p u,\phi\rangle_{X',X} = \int_0^T \langle -\Delta_p u(\tau),\phi(\tau)\rangle_p\,d\tau; \qquad u,\phi\in X
	\end{align*}
	is maximal monotone from $X=L^p(0,T;\W)$ to its dual $X'$.
\end{lem}
\begin{proof}
	It is sufficient (for instance, \cite[Thm. 2.4]{Barbu2010}) to demonstrate that $-\Delta_p$ is both monotone and hemicontinuous in the sense that $-\Delta_p(u+\lambda v) \to -\Delta_p u$ weakly in $X'$ as $\lambda\to 0$ for all $u,v\in X$.
	
	For monotonicity of $-\Delta_p:\W\to(\W)'$ , we record the identity
	\begin{align}\label{plaplacemaxl-2}
	\langle -\Delta_p u,u\rangle_p &= \int_\Omega |\grad u|^{p-2}\grad u\cdot \grad u \,dx + \int_\Gamma |u|^{p-2}u^2\,dS \notag\\
	& = ||\grad u||_p^p + |u|_p^p\notag\\
	& = ||u||_{1,p}^p
	\end{align}
	along with the estimate
	\begin{align}\label{plaplacemaxl-3}
	\langle -\Delta_p u,v\rangle_p &= \int_\Omega |\grad u|^{p-2}\grad u \cdot \grad v \,dx + \int_\Gamma |u|^{p-2}uv\,dS\notag\\
	&\leq\int_\Omega |\grad u|^{p-1}|\grad v|\,dx + \int_\Gamma |u|^{p-1}|v|\,dS \notag\\
	&\leq ||\grad u||_p^{p-1}||\grad v||_p + |u|_p^{p-1}|v|_p\notag \\
	& \leq  \frac{p-1}{p}||\grad u||_p^p + \frac{1}{p}||\grad v||_p^p + \frac{p-1}{p}|u|_p^p + \frac{1}{p}|v|_p^p\notag \\
	& = \frac{p-1}{p}||u||_{1,p}^p + \frac{1}{p}||v||_{1,p}^p	
	\end{align}
	obtained by H\"older and Cauchy's inequalities with conjugate exponents $p$ and $p/(p-1)$.
	Using linearity along with \eqref{plaplacemaxl-2} and \eqref{plaplacemaxl-3} we then have, as desired, 
	\begin{align*}
	\langle(-\Delta_p u) - (-\Delta_p v),u-v\rangle_p & = \langle -\Delta_p u,u\rangle_p -\langle -\Delta_p u,v\rangle_p - \langle -\Delta_p v,u\rangle_p + \langle -\Delta_p v,v\rangle_p\\
	& \geq ||u||_{1,p}^p + ||v||_{1,p}^p \\
	&\qquad - \frac{p-1}{p}||u||_{1,p}^p - \frac{1}{p}||v||_{1,p}^p - \frac{p-1}{p}||v||_{1,p}^p - \frac{1}{p}||u||_{1,p}^p \\
	&=0.
	\end{align*}

	Monotonicity of $-\Delta_p$ as a map from $X$ to $X'$ then follows immediately as
	\begin{align*}
	\langle (-\Delta_p u)-(-\Delta_p v),u-v\rangle_{X',X} = \int_0^T \langle (-\Delta_p u(\tau))-(-\Delta_p v(\tau)),u(\tau)-v(\tau)\rangle_p\,d\tau \geq 0
	\end{align*}
	for each $u,v\in X$ given that the integrand is nonnegative at each $\tau\in[0,T]$.
	
	For hemicontinuity we appeal to the dominated convergence theorem.  Given $u,v,\phi\in X$ and $|\lambda|\leq 1$ we may produce the pointwise upper bounds 
	\begin{subequations}
	\begin{align}\label{plaplacemax-ptbound}
	|\grad(u+\lambda v)|^{p-2}\grad(u+\lambda v)\cdot\grad\phi \leq (|\grad u| + |\grad v|)^{p-1}|\grad \phi|\text{ on }\Omega\times [0,T]
	\end{align}
	and
	\begin{align}\label{plaplacemax-ptbound2}
	|u+\lambda v|^{p-2}(u+\lambda v)\phi \leq (|u| + |v|)^{p-1}|\phi|\text{ on }\Gamma\times[0,T].
	\end{align}
	\end{subequations}
	Each of these pointwise bounds is indeed $L^1(0,T;\L[1])$ and $L^1(0,T;\Lb[1])$, respectively, since 
	\begin{multline*}
	\int_0^T \int_\Omega (|\grad u| + |\grad v|)^{p-1}|\grad \phi|\,dxd\tau \\
	\leq C\left(||\grad u||_{L^p(0,T;\L[p])}^{p-1} + ||\grad v||_{L^p(0,T;\L[p])}^{p-1}\right)||\grad\phi||_{L^p(0,T;\L[p])}<\infty
	\end{multline*}
	and, analagously,
	\begin{align*}
	\int_0^T \int_\Gamma (|u|+|v|)^{p-1}|\phi|\,dSd\tau <\infty
	\end{align*}
	from H\"older's inequality with conjugate exponents $p$ and $p/(p-1)$. Thus, 
		\begin{align*}
		\langle -\Delta_p (u+\lambda v),\phi\rangle_{X',X} &= \int_0^T \int_\Omega |\grad(u+\lambda v)|^{p-2}\grad(u+\lambda v)\cdot\grad \phi\,dxd\tau \\
		&\qquad\qquad +  \int_0^T \int_\Gamma |u+\lambda v|^{p-2}(u+\lambda v)\phi\,dSd\tau\\
		&\to \langle -\Delta_p u,\phi\rangle_{X',X} 	 
		\end{align*}
		as $\lambda \to 0$.
\end{proof}
} %

\begin{prop}[Proposition~A.1 in \cite{PRT-p-Laplacain}, similar in proof to {\cite[Prop. 3.2]{RW}} ]\label{prop:prodrule} Let $H$ be a Hilbert space and $X$ be a Banach space such that $X\into H \into X'$ where each injection is continuous with dense range.  If 
	\begin{alignat*}{4}
			f&\in&L^2(0,T;H),\qquad &g&\in& L^2(0,T;X),\\
			f'&\in&L^2(0,T;X'),\qquad&g'&\in&L^2(0,T;H), 	
	\end{alignat*}
	then $(f(t),h(t))_H$ coincides with an absolutely continuous function almost everywhere on $[0,T]$, and 
	\begin{align*}
	\frac{d}{dt}(f(t),g(t))_H = \langle f'(t),g(t)\rangle_{X',X} + (f(t),g'(t))_H,  \quad \text{a.e.}\,\, [0,T].
	\end{align*}
\end{prop}

\begin{lem}\label{lem:ll-trunc}If $f:\W \to\Lb[2]$ is locally Lipschitz continuous, then $f_K:\W \to \Lb[2]$ given by 
		\begin{align*}
		f_K(u)=\begin{cases}
		f(u),&\text{ for }||u||_{1,p} \leq K\\
		f\left(\frac{Ku}{||u||_{1,p}}\right),&\text{ for }|| u||_{1,p} >K
		\end{cases}
		\end{align*}
is globally Lipschitz for each $K>0$. Further, the global Lipschitz constant of $f_K$ is no more than twice the local Lipschitz constant of $f$ on the ball $B(0,K)\subset\W$. 
\end{lem}
\begin{proof}
	The proof given here is similar to Theorem~7.2 in \cite{CEL1} but does not require a Hilbert space structure on the domain of $f$.
	
	Let $C_K$ denote the local Lipschitz constant of $f$ on the ball $B(0,K)\subset \W$ and, for notational convenience, let $||\cdot||$ denote the $\W$ norm throughout.  For $u,v\in\W$ exactly one of three cases must apply based on the respective norms of these functions:
	\begin{enumerate}
		\item[Case 1:] $||u||,\;||v||\leq K$. Here, $f_K(u)=f(u)$ and $f_K(v)=f(v)$ so that \begin{align*}
		|f_K(u)-f_K(v)|_2\leq C_K ||u-v||.
		\end{align*}
		
		\item[Case 2:] $||u||\leq K,\; ||v||>K$. 
		Set $\tilde{v}=\frac{K}{||v||}v$ so that
		\begin{align*}
		|f_K(u)-f_K(v)|_2 &= |f(u)-f(\tilde{v})|_2 \leq C_K||u-\tilde{v}||.
		\end{align*}
		Noticing that $||v-\tilde{v}||=||v||-K$ we then have from the triangle inequality 
		\begin{align*}
		|f_K(u)-f_K(v)|_2&\leq C_K \left(||u-v||+||v-\tilde{v}||\right)\\
		&\leq C_K \left( ||u-v|| + \big| ||v||-||u|| \big| \right) \\
		&\leq 2C_K||u-v||.
		\end{align*}

		\item[Case 3:] $||u||,\;||v||>K$. Set $\tilde{u}=\frac{K}{||u||}u$ and $\tilde{v}$ as in case 2 so that 
		\begin{align*}
		||\tilde{u}-\tilde{v}|| &= K \left\Vert \frac{u}{||u||} - \frac{v}{||v||}\right\Vert\\
		&=\frac{K}{||u||||v||} \big\Vert u||v|| - v||u||\big\Vert\\
		&\leq \frac{K}{||u||||v||} \left(\big\Vert u||v|| - v||v||\big\Vert + \big\Vert v||v|| - v||u||\big\Vert \right)\\
		&\leq \frac{K}{||u||} \left( ||u-v|| + \big| ||v||-||u||\big| \right)\\
		&\leq 2||u-v||.
		\end{align*}
		As such, 
		\begin{align*}
		|f_K(u)-f_K(v)|_2\leq C_K ||\tilde{u}-\tilde{v}|| \leq 2C_K ||u-v||.
		\end{align*}
	\end{enumerate}
\end{proof}

\begin{proof}[Proof of Lemma~\ref{lem:sc-lip}] 
 Given any measurable functions $u,v$ define the regions
		\begin{flalign*}
		\Gamma_1 &= \{x\in\Gamma : |u(x)|,\;|v(x)|\leq 2n\},&
		\Gamma_2 &= \{x\in\Gamma : |u(x)|\leq 2n,\;|v(x)|> 2n\},\\
		\Gamma_3 &=  \{x\in\Gamma : |u(x)|,\;|v(x)|> 2n\},&
		\Gamma_4 &= \{x\in\Gamma : |u(x)|> 2n,\;|v(x)|\leq 2n\}
		\end{flalign*}
		so that $\Gamma = \bigcup_{i=1}^4 \Gamma_i$. For \eqref{sc-lip-1} in Lemma~\ref{lem:sc-lip}, bounds on the Lipschitz constant of $f_n$ may now be obtained by integration over each of the regions in turn.  Taking $u,v\in\W$, on $\Gamma_1$ the mean value theorem applied to $f_n$ asserts that 
		\begin{align}\label{sc-lip-mvt}
		|f_n(u)-f_n(v)|&\leq |f_n'(\xi_{u,v})||u-v| 
		\end{align}
		with $\xi_{u,v}\in[0,4n]$. Since $f$ and $f'$ are both bounded on this interval and 
		\begin{align*}
		f_n'(s)=f'(s)\eta_n(s)+f(s)\eta_n'(s)\leq f'(s)+\frac{C}{n}f(s)
		\end{align*}	
		we thus have $|f_n(u)-f_n(v)|\leq C_n|u-v|$ from \eqref{sc-lip-mvt} for a constant $C_n$ dependent upon $n$.  Hence, 
		\begin{align*}
		\int_{\Gamma_1}|f(u)-f(v)|^2\,dS \leq C_n\int_{\Gamma_1} |u-v|^2\,dS  \leq C_n||u-v||_{1,p}^2.
		\end{align*} 
		On $\Gamma_2$ note that $f_n(v)=\eta_n(v)=0$ whereupon \begin{align*}
		\int_{\Gamma_2} |f_n(u)-f_n(v)|^2\,dS = \int_{\Gamma_2}|f(u)|^2|\eta_n(u)-\eta_n(v)|^2\,dS.	
		\end{align*}
		Additionally utilizing the boundedness of $\eta_n'$ this time we obtain 
		\begin{align*}
		|\eta_n(u)-\eta_n(v)|\leq \frac{C}{n}|u-v|
		\end{align*}
		so that upon recalling that $|f(u)|$ is bounded on $\Gamma_2$, 
		\begin{align*}
		\int_{\Gamma_2}|f_n(u)-f_n(v)|^2\,dS\leq C_n\int_{\Gamma_2}|u-v|^2\,dS \leq C_n||u-v||_{1,p}^2.
		\end{align*}
		On $\Gamma_3$ we have $f_n(u)=f_n(v)=0$ and the case for $\Gamma_3$ is identical to that for $\Gamma_2$ with the roles of $u$ and $v$ reversed, thus completing the proof of \eqref{sc-lip-1}.
		
		For \eqref{sc-lip-2}, take $u,v\in\W[1-\epsilon,p]$ with $||u||_{1-\epsilon,p},\,||v||_{1-\epsilon,p}\leq R$. Since 
		\begin{align*}
		|f_n(u)-f_n(v)|&\leq |\eta_n(u)||f(u)-f(v)|+|f(v)||\eta_n(u)-\eta_n(v)|\\
		&\leq |f(u)-f(v)|+\frac{C}{n}(1+|v|^r)|u-v|
		\end{align*}
		from the triangle inequality, mean value theorem applied to $\eta_n$, and the bounds on $f$ and $\eta_n'$ we obtain 
		\begin{multline}\label{sc-lip-a2}
		|f_n(u)-f_n(v)|_{4/3}^{4/3} \leq C_R^{4/3}||u-v||_{1-\epsilon,p}^{4/3} + \frac{C}{n^{4/3}}\int_{\Gamma}(1+|v|^{4r/3})|u-v|^{4/3}\,dS
		\end{multline}
		with $C_R$ provided from Lemma~\ref{lem:f-lipshitz}. In order to estimate the second term in \eqref{sc-lip-a2} we again consider the four regions defined at the onset of the proof. On $\Gamma_1$ it is the case that $|v|\leq 2n$ so that  $n^{-4/3}\leq C|v|^{-4/3}$. As such, 
		\begin{align*}
		\frac{C}{n^{4/3}}(1+|v|^{4r/3})|u-v|^{4/3}
		& \leq C(1+|v|^r)|u-v|^{4/3}.
		\end{align*}
		Moreover, this same bound can be produced on $\Gamma_4$ where  $v$ enjoys the same bound as on $\Gamma_1$, and then subsequently on $\Gamma_2$ by reversing the roles of $u$ and $v$.  Since $f_n$ is identically zero on $\Gamma_3$,  we may thus bound the final term in \eqref{sc-lip-a2} as 
		\begin{align*}
		\frac{C}{n^{4/3}}\int_{\Gamma}(1+|v|^{4r/3})|u-v|^{4/3}\,dS
		& \leq C\int_{\Gamma}(1+|v|^r)|u-v|^{4/3}\,dS\\
		&\leq C\left(1+|v|_{3r/2}^r\right) |u-v|_{4}^{4/3}\\
		&\leq C(1+||v||_{1-\epsilon}^r)||u-v||_{1-\epsilon}^{4/3}
		\end{align*}
		when recalling the trace mappings $\W[1-\epsilon,p]\tinto \Lb[4]$ and $\W[1-\epsilon,p]\tinto \Lb[3r/2]$ for sufficiently small $\epsilon$. This establishes that, in fact, 
		\begin{align*}
		|f_n(u)-f_n(v)|_{4/3} \leq C_R ||u-v||_{1-\epsilon,p}
		\end{align*}
		for a constant $C_R$ not dependent upon $n$.
\end{proof}

\bibliographystyle{abbrv}
\bibliography{../mohnick}

\def\cprime{$'$}
\begin{thebibliography}{10}

\bibitem{ADAMS}
R.~A. Adams.
\newblock {\em Sobolev Spaces}.
\newblock Academic Press, New York, 1975.

\bibitem{AR2}
K.~Agre and M.~A. Rammaha.
\newblock Systems of nonlinear wave equations with damping and source terms.
\newblock {\em Differential Integral Equations}, 19(11):1235--1270, 2006.

\bibitem{MR797319}
P.~Aviles and J.~Sandefur.
\newblock Nonlinear second order equations with applications to partial
  differential equations.
\newblock {\em J. Differential Equations}, 58(3):404--427, 1985.

\bibitem{Barbu2010}
V.~Barbu.
\newblock {\em Nonlinear differential equations of monotone types in {B}anach
  spaces}.
\newblock Springer, 2010.

\bibitem{BM2}
A.~Benaissa and S.~Mokeddem.
\newblock Decay estimates for the wave equation of {$p$}-{L}aplacian type with
  dissipation of {$m$}-{L}aplacian type.
\newblock {\em Math. Methods Appl. Sci.}, 30(2):237--247, 2007.

\bibitem{Bia:95:NA}
A.~C. Biazutti.
\newblock On a nonlinear evolution equation and its applications.
\newblock {\em Nonlinear Anal.}, 24(8):1221--1234, 1995.

\bibitem{BL3}
L.~Bociu and I.~Lasiecka.
\newblock Blow-up of weak solutions for the semilinear wave equations with
  nonlinear boundary and interior sources and damping.
\newblock {\em Appl. Math. (Warsaw)}, 35(3):281--304, 2008.

\bibitem{BL2}
L.~Bociu and I.~Lasiecka.
\newblock Uniqueness of weak solutions for the semilinear wave equations with
  supercritical boundary/interior sources and damping.
\newblock {\em Discrete Contin. Dyn. Syst.}, 22(4):835--860, 2008.

\bibitem{BL1}
L.~Bociu and I.~Lasiecka.
\newblock Local {H}adamard well-posedness for nonlinear wave equations with
  supercritical sources and damping.
\newblock {\em J. Differential Equations}, 249(3):654--683, 2010.

\bibitem{Boyer2013}
F.~Boyer and P.~Fabrie.
\newblock {\em Mathematical tools for the study of the incompressible
  {N}avier-{S}tokes equations and related models}.
\newblock Springer, 2013.

\bibitem{CCL}
M.~M. Cavalcanti, V.~N. Domingos~Cavalcanti, and I.~Lasiecka.
\newblock Well-posedness and optimal decay rates for the wave equation with
  nonlinear boundary damping---source interaction.
\newblock {\em J. Differential Equations}, 236(2):407--459, 2007.

\bibitem{MR1634008}
F.~Chen, B.~Guo, and P.~Wang.
\newblock Long time behavior of strongly damped nonlinear wave equations.
\newblock {\em J. Differential Equations}, 147(2):231--241, 1998.

\bibitem{CEL1}
I.~Chueshov, M.~Eller, and I.~Lasiecka.
\newblock On the attractor for a semilinear wave equation with critical
  exponent and nonlinear boundary dissipation.
\newblock {\em Comm. Partial Differential Equations}, 27(9-10):1901--1951,
  2002.

\bibitem{GT}
V.~Georgiev and G.~Todorova.
\newblock Existence of a solution of the wave equation with nonlinear damping
  and source terms.
\newblock {\em J. Differential Equations}, 109(2):295--308, 1994.

\bibitem{MR1112054}
J.-M. Ghidaglia and A.~Marzocchi.
\newblock Longtime behaviour of strongly damped wave equations, global
  attractors and their dimension.
\newblock {\em SIAM J. Math. Anal.}, 22(4):879--895, 1991.

\bibitem{G}
R.~T. Glassey.
\newblock Blow-up theorems for nonlinear wave equations.
\newblock {\em Math. Z.}, 132:183--203, 1973.

\bibitem{GR1}
Y.~Guo and M.~A. Rammaha.
\newblock Blow-up of solutions to systems of nonlinear wave equations with
  supercritical sources.
\newblock {\em Appl. Anal.}, 92(6):1101--1115, 2013.

\bibitem{GR2}
Y.~Guo and M.~A. Rammaha.
\newblock Global existence and decay of energy to systems of wave equations
  with damping and supercritical sources.
\newblock {\em Z. Angew. Math. Phys.}, 64(3):621--658, 2013.

\bibitem{GR}
Y.~Guo and M.~A. Rammaha.
\newblock Systems of nonlinear wave equations with damping and supercritical
  boundary and interior sources.
\newblock {\em Trans. Amer. Math. Soc.}, 366(5):2265--2325, 2014.

\bibitem{GRSTT}
Y.~Guo, M.~A. Rammaha, S.~Sakuntasathien, E.~S. Titi, and D.~Toundykov.
\newblock Hadamard well-posedness for a hyperbolic equation of viscoelasticity
  with supercritical sources and damping.
\newblock {\em J. Differential Equations}, 257(10):3778--3812, 2014.

\bibitem{KL}
H.~Koch and I.~Lasiecka.
\newblock Hadamard well-posedness of weak solutions in nonlinear dynamic
  elasticity-full von {K}arman systems.
\newblock In {\em Evolution equations, semigroups and functional analysis
  (Milano, 2000)}, volume~50 of {\em Progr. Nonlinear Differential Equations
  Appl.}, pages 197--216. Birkh\"auser, Basel, 2002.

\bibitem{l1}
H.~A. Levine.
\newblock Instability and nonexistence of global solutions to nonlinear wave
  equations of the form {$Pu\sb{tt}=-Au+\mathcal{F}(u)$}.
\newblock {\em Trans. Amer. Math. Soc.}, 192:1--21, 1974.

\bibitem{LM1}
J.-L. Lions and E.~Magenes.
\newblock {\em Non-homogeneous boundary value problems and applications. {V}ol.
  {I}}.
\newblock Springer-Verlag, New York, 1972.

\bibitem{MR0199519}
J.-L. Lions and W.~A. Strauss.
\newblock Some non-linear evolution equations.
\newblock {\em Bull. Soc. Math. France}, 93:43--96, 1965.

\bibitem{LSt}
J.-L. Lions and W.~A. Strauss.
\newblock Some non-linear evolution equations.
\newblock {\em Bull. Soc. Math. France}, 93:43--96, 1965.

\bibitem{nak-nan:75}
M.~Nakao and T.~Nanbu.
\newblock Existence of global (bounded) solutions for some nonlinear evolution
  equations of second order.
\newblock {\em Math. Rep. College General Ed. Kyushu Univ.}, 10(1):67--75,
  1975.

\bibitem{PRT-p-Laplacain}
P.~Pei, M.~A. Rammaha, and D.~Toundykov.
\newblock Weak solutions and blow-up for wave equations of {$p$}-{L}aplacian
  type with supercritical sources.
\newblock {\em J. Math. Phys.}, 56(8):081503, 30, 2015.

\bibitem{Radu1}
P.~Radu.
\newblock Weak solutions to the initial boundary value problem for a semilinear
  wave equation with damping and source terms.
\newblock {\em Appl. Math. (Warsaw)}, 35(3):355--378, 2008.

\bibitem{RTW}
M.~Rammaha, D.~Toundykov, and Z.~Wilstein.
\newblock Global existence and decay of energy for a nonlinear wave equation
  with {$p$}-{L}aplacian damping.
\newblock {\em Discrete Contin. Dyn. Syst.}, 32(12):4361--4390, 2012.

\bibitem{RW}
M.~A. Rammaha and Z.~Wilstein.
\newblock Hadamard well-posedness for wave equations with p-{L}aplacian damping
  and supercritical sources.
\newblock {\em Adv. Differential Equations}, 17(1-2):105--150, 2012.

\bibitem{V1}
E.~Vitillaro.
\newblock Some new results on global nonexistence and blow-up for evolution
  problems with positive initial energy.
\newblock {\em Rend. Istit. Mat. Univ. Trieste}, 31(suppl. 2):245--275, 2000.
\newblock Workshop on Blow-up and Global Existence of Solutions for Parabolic
  and Hyperbolic Problems (Trieste, 1999).

\bibitem{V3}
E.~Vitillaro.
\newblock Global existence for the wave equation with nonlinear boundary
  damping and source terms.
\newblock {\em J. Differential Equations}, 186(1):259--298, 2002.

\bibitem{V2}
E.~Vitillaro.
\newblock A potential well theory for the wave equation with nonlinear source
  and boundary damping terms.
\newblock {\em Glasg. Math. J.}, 44(3):375--395, 2002.

\bibitem{MR586981}
G.~F. Webb.
\newblock Existence and asymptotic behavior for a strongly damped nonlinear
  wave equation.
\newblock {\em Canad. J. Math.}, 32(3):631--643, 1980.

\end{thebibliography}
\end{document}